%% file: markov.tex
\documentclass[a4paper, 11pt]{article}

\usepackage[english]{babel}
\usepackage{array} % for tabular environment
\usepackage{mathtools} % for dcases environment
\usepackage{amsthm} 
\usepackage{graphicx}
\usepackage{rotating}
\usepackage{pinlabel}
\usepackage{tikz}
\usetikzlibrary{trees}
\usetikzlibrary{cd}
\usepackage[alwaysadjust]{paralist}
\usepackage[%
breaklinks,
colorlinks=true,
linkcolor=black,
anchorcolor=black,
citecolor=black,
filecolor=black,
menucolor=black,
urlcolor=black]{hyperref}
\usepackage[margin=\parindent]{caption}
\usepackage[charter]{mathdesign}
\usepackage[textwidth=360pt, textheight=594pt, vmarginratio=1:1]{geometry}

\newcommand{\R}{\mathbb{R}}
\newcommand{\Q}{\mathbb{Q}}
\newcommand{\C}{\mathbb{C}}
\newcommand{\Z}{\mathbb{Z}}
\newcommand{\im}{\operatorname{Im}}
\newcommand{\group}[1]{\mathit{#1}}
\newcommand{\GLTZ}{\group{GL}_2(\Z)}
\newcommand{\PSLTZ}{\group{PSL}_2(\Z)}
\newcommand{\GLTR}{\group{GL}_2(\R)}
\newcommand{\PGLTR}{\group{PGL}_2(\R)}
\newcommand{\SLTR}{\group{SL}_2(\R)}
\newcommand{\PSLTR}{\group{PSL}_2(\R)}
\newcommand{\Id}{\mathit{Id}}
\newcommand{\Isom}{\group{Isom}}

\theoremstyle{plain}
\newtheorem*{theorem}{Theorem}
\newtheorem{proposition}{Proposition}[section]
\newtheorem{lemma}[proposition]{Lemma}

\theoremstyle{definition}

\newtheorem{problem}[proposition]{Problem}
\newtheorem{remark}[proposition]{Remark}

\title{The hyperbolic geometry of Markov's theorem on Diophantine
  approximation and quadratic forms}
\author{Boris Springborn}
\date{}

\begin{document}

\maketitle

\begin{abstract}
  Markov's theorem classifies the worst irrational numbers with
  respect to rational approximation and the indefinite binary
  quadratic forms whose values for integer arguments stay farthest
  away from zero. The main purpose of this paper is to present a new
  proof of Markov's theorem using hyperbolic geometry. The main
  ingredients are a dictionary to translate between hyperbolic
  geometry and algebra/number theory, and some very basic tools
  borrowed from modern geometric Teichm{\"u}ller theory. Simple closed
  geodesics and ideal triangulations of the modular torus play an
  important role, and so do the problems: How far can a straight line
  crossing a triangle stay away from the vertices? How far can it stay
  away from the vertices of the tessellation generated by the
  triangle? Definite binary quadratic forms are briefly discussed in
  the last section.

  \vspace{0.5\baselineskip}\noindent%
  \emph{MSC (2010).} 11J06, 32G15\\
  \emph{Key words and phrases.} Diophantine approximation, 
  quadratic form,
  modular torus,
  closed geodesic
\end{abstract}

\section{Introduction}
\label{sec:intro}

The main purpose of this article is to present a new proof of Markov's
theorem~\cite{markoff79, markoff80} (Secs.~\ref{sec:mostirrational},
\ref{sec:forms}) using hyperbolic geometry. Roughly, the following
dictionary is used to translate between hyperbolic geometry and
algebra/number theory:

\medskip
\begin{center}
  \small
  \newlength{\widthone}
  \newlength{\widthtwo}
  \newlength{\widththree}
  \settowidth{\widthone}{signed distance between horocycles}
  \settowidth{\widthtwo}{indefinite binary quadratic form $f$}
  \settowidth{\widththree}{Sec.~\ref{sec:point}}
  \renewcommand{\arraystretch}{1.5}
  \begin{tabular}{m{\widthone}|m{\widthtwo}|m{\widththree}}
    \hfill Hyperbolic Geometry \hspace*{\fill} 
    & \hfill Algebra/Number Theory\hspace*{\fill} 
    & 
    \\
    \hline
    horocycle 
    & nonzero vector $(p,q)\in\R^{2}$
    & Sec.~\ref{sec:horo}
    \\
    geodesic 
    & indefinite binary quadratic form $f$
    & Sec.~\ref{sec:geodesics}
    \\
    point 
    & definite binary quadratic form $f$
    & Sec.~\ref{sec:point}
    \\
    signed distance between horocycles
    & $2\log
      \Big|\det\Big(
      \begin{smallmatrix}
        p_{1} & p_{2}\\
        q_{1} & q_{2}
      \end{smallmatrix}
                \Big)\Big|$
    & \eqref{eq:dhh}
    \\
    signed distance between horocycle and geodesic/point
    & $\displaystyle\log\frac{f(p,q)}{\sqrt{|\det f|}}$
    & \eqref{eq:dhgf} \eqref{eq:dhzf}
    \\
    ideal triangulation of the modular torus
    & Markov triple
    & Sec.~\ref{sec:triangulations}
  \end{tabular}
\end{center}
\medskip

The proof is based on Penner's geometric interpretation of Markov's
equation~\cite[p.~335f]{penner87} (Sec.~\ref{sec:triangulations}), and
the main tools are borrowed from his theory of decorated
Teichm{\"u}ller space (Sec.~\ref{sec:ideal_triang}). Ultimately, the
proof of Markov's theorem boils down to the question:

\begin{quotation}\noindent
  How far can a straight line crossing a triangle stay away from all
  vertices?
\end{quotation}

It is fun and a recommended exercise to consider this question in
elementary euclidean geometry. Here, we need to deal with ideal
hyperbolic triangles, decorated with horocycles at the vertices, and
``distance from the vertices'' is to be understood as ``signed
distance from the horocycles'' (Sec.~\ref{sec:crossing}).

The subjects of this article, Diophantine approximation, quadratic
forms, and the hyperbolic geometry of numbers, are connected with
diverse areas of mathematics and its applications, ranging from from
the phyllotaxis of plants~\cite{conway96} to the stability of the
solar system~\cite{hubbard07}, and from Gauss' \emph{Disquisitiones
  Arithmeticae} to Mirzakhani's Fields
Medal~\cite{mirzakhani08}. An adequate survey of this area, even if
limited to the most important and most recent contributions, would be
beyond the scope of this introduction. The books by
Aigner~\cite{aigner13} and Cassels~\cite{cassels57} are excellent
references for Markov's theorem, Bombieri~\cite{bombieri07} provides a
concise proof, and more about the Markov and Lagrange spectra can be
found in Malyshev's survey~\cite{maly77} and the book by Cusick and
Flahive~\cite{cusick89}. The following discussion focuses on a few
historic sources and the most immediate context and is far from
comprehensive.

One can distinguish two approaches to a geometric treatment of
continued fractions, Diophantine approximation, and quadratic
forms. In both cases, number theory is connected to geometry by a
common symmetry group, $\GLTZ$. The first approach, known as the
geometry of numbers and connected with the name of Minkowski, deals
with the geometry of the $\Z^{2}$-lattice. Klein interpreted continued
fraction approximation, intuitively speaking, as ``pulling a thread
tight'' around lattice points~\cite{klein95,klein96}. This approach
extends naturally to higher dimensions, leading to a multidimensional
generalization of continued fractions that was championed by
Arnold~\cite{arnold98,arnold01}. Delone's comments on Markov's
work~\cite{delone05} also belong in this category (see
also~\cite{fuchs07}).

In this article, we pursue the other approach involving Ford circles
and the Farey tessellation of the hyperbolic plane
(Fig.~\ref{fig:farey}). This approach could be called the hyperbolic
geometry of numbers. Before Ford's geometric proof~\cite{ford38} of
Hurwitz's theorem~\cite{hurwitz91} (Sec.~\ref{sec:mostirrational}),
Speiser had apparently used the Ford circles to prove a weaker
approximation theorem. However, only the following note survives of
his talk~\cite[my translation]{speiser23}:

\begin{quotation}
  \noindent\small
  \emph{A geometric figure related to number theory.} If one
  constructs in the upper half plane for every rational point of the
  $x$-axis with abscissa $\frac{p}{q}$ the circle of radius
  $\frac{1}{2q^{2}}$ that touches this point, then these circles
  do not overlap anywhere, only tangencies occur. The domains that are not
  covered consist of circular triangles. Following the line
  $x=\omega$ (irrational number) downward towards the $x$-axis, one
  intersects infinitely many circles, i.e., the inequality
  \begin{equation*}
    \big|\omega-\frac{p}{q}\big|<\frac{1}{2q^{2}}
  \end{equation*}
  has infinitely many solutions. They constitute the approximations by
  Min\-kow\-ski's continued fractions.

  If one increases the radii to $\frac{1}{\sqrt{3}q^2}$, then the gaps
  close and one obtains the theorem on the maximum of positive binary
  quadratic forms.
\end{quotation}

See Rem.~\ref{rem:vahlen} and Sec.~\ref{sec:point} for brief comments
on these theorems. Based on Speiser's talk,
Z{\"u}llig~\cite{zuellig28} developed a comprehensive geometric theory
of continued fractions, including a geometric proof of Hurwitz's
theorem.

Both Z{\"u}llig and Ford treat the arrangement of Ford circles using
elementary euclidean geometry and do not mention any connection with
hyperbolic geometry. In Sec.~\ref{sec:hurwitz}, we transfer their
proof of Hurwitz's theorem to hyperbolic geometry. The conceptual
advantage is obvious: One has to consider only three circles instead
of infinitely many, because all triples of pairwise touching
horocycles are congruent.

Today, the role of hyperbolic geometry is well understood. Continued
fraction expansions encode directions for navigating the Farey
tessellation of the hyperbolic plane~\cite{bonahon09,hatcher,
  series85a}. In fact, much was already known to
Hurwitz~\cite{hurwitz94} and Klein~\cite{klein90, klein96}. According
to Klein~\cite[p.~248]{klein96}, they built on
Hermite's~\cite{hermite51} purely algebraic discovery of an invariant
``incidence'' relation between definite and indefinite forms, which
they translated into the language of geometry. While Hurwitz and Klein
never mention horocycles, they knew the other entries of the
dictionary, and even use the Farey triangulation. In the Cayley--Klein
model of hyperbolic space, the geometric interpretation of binary
quadratic forms is easily established: The projectivized vector space
of real binary quadratic forms is a real projective plane and the
degenerate forms are a conic section. Definite forms correspond to
points inside this conic, hence to points of the hyperbolic plane,
while indefinite forms correspond to points outside, hence, by
polarity, to hyperbolic lines. From this geometric point of view,
Klein and Hurwitz discuss classical topics of number theory like the
reduction of binary quadratic forms, their automorphisms, and the role
of Pell's equation. Strangely, it seems they never treated Diophantine
approximation or Markov's work this way.

Cohn~\cite{cohn55} noticed that Markov's Diophantine
equation~\eqref{eq:markov} can easily be obtained from an elementary
identity of Fricke involving the traces of $2\times 2$-matrices. Based
on this algebraic coincidence, he developed a geometric interpretation
of Markov forms as simple closed geodesics in the modular
torus~\cite{cohn71,cohn72}, which is also adopted in this
article.

A much more geometric interpretation of Markov's equation was
discovered by Penner (as mentioned above), as a byproduct of his
decorated Teichm{\"u}ller theory~\cite{penner87,penner12}. This
interpretation focuses on ideal triangulations of the modular torus,
decorated with a horocycle at the cusp, and the weights of their edges
(Sec.~\ref{sec:triangulations}). Penner's interpretation also explains
the role of simple closed geodesics (Sec.~\ref{sec:arcs_geodesics}).

Markov's original proof (see~\cite{bombieri07} for a concise modern
exposition) is based on an analysis of continued fraction
expansions. Using the interpretation of continued fractions as
directions in the Farey tessellation mentioned above, one can
translate Markov's proof into the language of hyperbolic geometry. The
analysis of allowed and disallowed subsequences in an expansion
translates to symbolic dynamics of geodesics~\cite{series85}.

In his 1953 thesis, which was published much later,
Gorshkov~\cite{gorshkov77} provided a genuinely new proof of Markov's
theorem using hyperbolic geometry. It is based on two important ideas
that are also the foundation for the proof presented here. First,
Gorshkov realized that one should consider all ideal triangulations of
the modular torus, not only the projected Farey tessellation. This
reduces the symbolic dynamics argument to almost nothing (in this
article, see Proposition~\ref{prop:markov_forms_geo}, the proof of
implication ``$\text{(c)}\Rightarrow\text{(a)}$''). Second, he
understood that Markov's theorem is about the distance of a geodesic
to the vertices of a triangulation. However, lacking modern geometric
tools of Teichm{\"u}ller theory (like horocycles), Gorshkov was not
able to treat the geometry of ideal triangulations directly. Instead,
he considers compact tori composed of two equilateral hyperbolic
triangles and lets the side length tend to infinity. The compact tori
have a cone-like singularity at the vertex, and the developing map
from the punctured torus to the hyperbolic plane has infinitely many
sheets. This limiting process complicates the argument
considerably. Also, the trigonometry becomes simpler when one needs to
consider only decorated ideal triangles. Gorshkov's decision ``not to
restrict the exposition to the minimum necessary for proving Markov's
theorem but rather to execute it with considerable completeness,
retaining everything that is of independent interest'' makes it harder
to recognize the main lines of argument. This, together with an unduly
dismissive MathSciNet review, may account for the lack of recognition
his work received.

In this article, we adopt the opposite strategy and stick to proving
Markov's theorem. Many natural generalizations and related topics are
beyond the scope of this paper, for example the approximation of
complex numbers~\cite{dani14,ford18,ford25,schmidt75}, generalizations
to other Riemann surfaces or discrete
groups~\cite{abe13,beardon86,bowditch98,haas86,lehner84,schmidt76,schmidt77},
higher dimensional manifolds~\cite{hersonsky02,vulakh99}, other
Diophantine approximation theorems, for example
Khinchin's~\cite{sullivan82}, and the asymptotic growth of Markov
numbers and lengths of closed
geodesics~\cite{bowditch96,mcshane91,mcshane95,spalding16,spalding17,zagier82}. Is
the treatment of Markov's equation using
$3\times 3$-matrices~\cite{perrine09,riedel12} related? Do the methods
presented here help to cover a larger part of the Markov and Lagrange
spectra by considering more complicated
geodesics~\cite{crisp98,crisp93a,crisp93b}?  Can one treat, say,
ternary quadratic forms or binary cubic forms in a similar fashion?

The notorious Uniqueness Conjecture for Markov numbers
(Rem.~\ref{rem:markov} (iv)), which goes back to a neutral statement
by Frobenius~\cite[p.~461]{frobenius13}, says in geometric terms: If
two simple closed geodesics in the modular torus have the same length,
then they are related by an isometry of the modular
torus~\cite{schmutz98}. Equivalently, if two ideal arcs have the same
weight, they are related this way. Hyperbolic geometry was
instrumental in proving the uniqueness conjecture for Markov numbers
that are prime powers~\cite{button98,lang07,schmutz96}. Will geometry
also help to settle the full Uniqueness Conjecture, or is it ``a
conjecture in pure number theory and not tractable by hyperbolic
geometry arguments''~\cite{mcshane08}?  Will combinatorial methods
succeed? Who knows. These may not even be very meaningful questions,
like asking: ``Will a proof be easier in English, French, Russian, or
German?'' On the other hand, sometimes it helps to speak more than one
language.

\section{The worst irrational numbers}
\label{sec:mostirrational}

There are two versions of Markov's theorem. One deals with Diophantine
approximation, the other with quadratic forms. In this section, we
recall some related theorems and state the Diophantine approximation
version in the form in which we will prove it (Sec.~\ref{sec:markovproof}).
The following section is about the quadratic forms version.

Let $x$ be an irrational number. For every positive integer $q$ there
is obviously a fraction $\frac{p}{q}$ that approximates $x$ with error
less than $\smash[b]{\frac{1}{2q}}$. If one chooses denominators more
carefully, one can find a sequence of fractions converging to $x$ with
error bounded by $\smash[t]{\frac{1}{q^{2}}}$:

\begin{theorem}
  For every irrational number $x$, there are infinitely many
  fractions $\frac{p}{q}$ satisfying 
  \begin{equation*}
    \Big|x-\frac{p}{q}\Big|<\frac{1}{q^{2}}\,.
  \end{equation*}
\end{theorem}

This theorem is sometimes attributed to Dirichlet although the
statement had ``long been known from the theory of continued
fractions''~\cite{dirichlet42}. In fact, Dirichlet provided a
particularly simple proof of a multidimensional generalization, using
what later became known as the pigeonhole principle.

Klaus Roth was awarded a Fields Medal in 1958 for showing that the
exponent $2$ in Dirichlet's approximation theorem is optimal~\cite{roth55}: 

\begin{theorem}[Roth]
  Suppose $x$ and $\alpha$ are real numbers, $\alpha>2$. If there are infinitely
  many reduced fractions $\frac{p}{q}$ satisfying
  \begin{equation*}
    \Big|x-\frac{p}{q}\Big|<\frac{1}{q^{\alpha}}\,,
  \end{equation*}
  then $x$ is transcendental.
\end{theorem}

In other words, if the exponent in the error bound is greater than $2$
then algebraic irrational numbers cannot be approximated. This is an
example of a general observation: ``From the point of view of rational
approximation, \emph{the simplest numbers are the worst}'' (Hardy \&
Wright~\cite{hardy08}, p.~209, their emphasis). Roth's theorem shows
that the worst irrational numbers are algebraic. Markov's theorem,
which we will state shortly, shows that the worst algebraic
irrationals are quadratic.

While the exponent is optimal, the constant factor in Dirichlet's
approximation theorem can be improved. Hurwitz~\cite{hurwitz91}
showed that the optimal constant is $\frac{1}{\sqrt{5}}$, and that the
golden ratio belongs to the class of very worst irrational numbers:

\begin{theorem}[Hurwitz]
  \label{thm:hurwitz}
  (i) For every irrational number $x$, there are infinitely many
  fractions $\frac{p}{q}$ satisfying
  \begin{equation}
    \label{eq:hurwitz}
    \Big|x-\frac{p}{q}\Big| < \frac{1}{\sqrt{5}\,q^{2}}\,.
  \end{equation}
  (ii) If $\lambda>\sqrt{5}$, and if $x$ is equivalent to the golden
  ratio $\phi=\frac{1}{2}(1+\sqrt{5})$, then there are only finitely
  many fractions~$\frac{p}{q}$ satisfying
  \begin{equation}
    \label{eq:lambda}
    \Big|x-\frac{p}{q}\Big| < \frac{1}{\lambda\,q^{2}}\,.
  \end{equation}
\end{theorem}

Two real numbers $x$, $x'$ are called \emph{equivalent} if
\begin{equation}
  \label{eq:equiv_number}
  x'=\frac{ax+b}{cx+d},
\end{equation}
for some integers $a$, $b$, $c$, $d$ satisfying
\begin{equation*}
  |ad-bc|=1.
\end{equation*}
If infinitely many fractions satisfy~\eqref{eq:lambda} for some $x$,
then the same is true for any equivalent number $x'$. This follows
simply from the identity
\begin{equation*}
  (q')^{2}\,\Big|x'-\frac{p'}{q'}\Big| = q^{2}\,\Big|x-\frac{p}{q}\Big|\;
  \frac{\big|c\,\big(\frac{p}{q}\big)+d\big|}{\big|cx+d\big|},
\end{equation*}
where $x$ and $x'$ are related by \eqref{eq:equiv_number} and
$p'=ap+bq$, $q'=cp+dq$. (Note that the last factor on the right hand
side tends to $1$ as $\frac{p}{q}$ tends to $x$.)

Hurwitz also states the following results, ``whose proofs can easily
be obtained from Markov's investigation'' of indefinite quadratic forms:

$\bullet$ If~$x$ is an irrational number \emph{not} equivalent to the golden
ratio $\phi$, then infinitely many fractions
satisfy~\eqref{eq:lambda} with
$\lambda=2\sqrt{2}$. 

$\bullet$ For any $\lambda<3$, there are only finitely many
equivalence classes of numbers that cannot be approximated, i.e., for
which there are only finitely many fractions
satisfying~\eqref{eq:lambda}. But for $\lambda=3$, there
are infinitely many classes that cannot be approximated.

Hurwitz stops here, but the story continues. Table~\ref{tab:numbers}
lists representatives $x$ of the five worst classes of irrational
numbers, and the largest values $L(x)$ for $\lambda$ for which there
exist infinitely many fractions satisfying~\eqref{eq:lambda}. For
example, $\sqrt{2}$ belongs to the class of second worst irrational
numbers. The last two columns will be explained in the statement of
Markov's theorem.

\begin{table}[h]
  \centering
  \renewcommand{\arraystretch}{1.5}
  \begin{tabular}{c|
                  >{$}c<{$}|
                  >{$}r<{$}@{$\;=\;$}>{$}l<{$}|
                  >{$}c<{$}@{ }>{$}c<{$}@{ }>{$}c<{$}|
                  >{$}c<{$}@{ }>{$}c<{$}}
    rank & x & \multicolumn{2}{c|}{$L(x)$} & \ a\  & \ b\  & \ c\  & \ p_{1}\  &
                                                                       \ p_{2}\ \\
    \hline
    1
    &\frac{1}{2}(1+\sqrt{5})      
    & \sqrt{5} & 2.2\ldots   
    & 1 & 1 & 1
    & 0 & 1
    \\
    2
    &\sqrt{2}
    & 2\sqrt{2}
    & 2.8\ldots 
    & 1 & 1 & 2
    & -1 & 1
    \\
    3
    &\frac{1}{10}(9+\sqrt{221})
    & \frac{1}{5}\sqrt{221} & 2.97\ldots
    & 1 & 2 & 5
    & -1 & 2
    \\
    4
    &\frac{1}{26}(23+\sqrt{1517}) 
    & \frac{1}{13}\sqrt{1517} & 2.996\ldots
    & 1 & 5 & 13
    & -3 & 2
    \\
    5
    &\frac{1}{58}(5+\sqrt{7565})  
    & \frac{1}{29}\sqrt{7565} & 2.9992\ldots
    & 2 & 5 & 29
    & -7 & 3
  \end{tabular}
\caption{The five worst classes of irrational numbers}
\label{tab:numbers}
\end{table}

Markov's theorem establishes an explicit bijection between the
equivalence classes of the worst irrational numbers, and sorted Markov
triples. Here, \emph{worst irrational numbers} means precisely those
that cannot be approximated for some $\lambda<3$. A \emph{Markov
  triple} is a triple $(a,b,c)$ of positive integers satisfying
Markov's equation
\begin{equation}
  \label{eq:markov}
  a^{2}+b^{2}+c^{2}=3abc.
\end{equation}
A \emph{Markov number} is a number that appears in some Markov triple.
Any permutation of a Markov triple is also a Markov triple. A
\emph{sorted Markov triple} is a Markov triple $(a,b,c)$ with
$a\leq b\leq c$. 

We review some basic facts about Markov triples and refer to the
literature for details, for example~\cite{aigner13,cassels57}. First
and foremost, note that Markov's equation~\eqref{eq:markov} is
quadratic in each variable. This allows one to generate new solutions
from known ones: If $(a,b,c)$ is a Markov triple, then so are its
\emph{neighbors}
\begin{equation}
\label{eq:primed_triples}
(a',b,c),\quad 
(a,b',c),\quad
(a,b,c'), 
\end{equation}
where
\begin{equation}
  \label{eq:a_prime}
  a' = 3bc-a = \frac{b^{2}+c^{2}}{a},
\end{equation}
and similarly for $b'$ and $c'$. Hence, there are three involutions
$\sigma_{k}$ on the set of Markov triples that map any triple
$(a,b,c)$ to its neighbors:
\begin{equation}
  \label{eq:sigma_k}
  \sigma_{1}(a,b,c)=(a',b,c),\quad
  \sigma_{2}(a,b,c)=(a,b',c),\quad
  \sigma_{3}(a,b,c)=(a,b,c').  
\end{equation}
 These
involutions act without fixed points and every Markov triple can be
obtained from a single Markov triple, for example from $(1,1,1)$, by
applying a composition of these involutions. The sequence of involutions
is uniquely determined if one demands that no triple is visited
twice. Thus, the solutions of Markov's equation~\eqref{eq:markov} form
a trivalent tree, called the \emph{Markov tree}, with Markov triples as vertices and
edges connecting neighbors (see Fig.~\ref{fig:markovtree}).
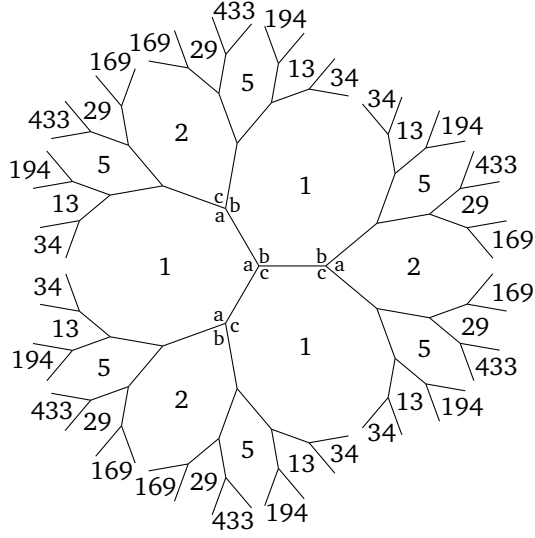
\begin{figure}
  \centering
  \input{markovtree}
  \caption{Markov tree}
  \label{fig:markovtree}
\end{figure}

\begin{theorem}[Markov, Diophantine approximation version]

  (i) Let $(a,b,c)$ be any Markov triple, let $p_{1}$, $p_{2}$ be
  integers satisfying
  \begin{equation}
    \label{eq:p1p2}
    p_{2}b-p_{1}a=c,
  \end{equation}
  and let
  \begin{equation}
    \label{eq:markov_x}
    x = \frac{p_{2}}{a}+\frac{b}{ac}-\frac{3}{2} + \sqrt{\frac{9}{4}-\frac{1}{c^{2}}}\,.
  \end{equation}
  Then there are infinitely many fractions $\frac{p}{q}$
  satisfying~\eqref{eq:lambda} with
  \begin{equation}
    \label{eq:markov_lambda}
    \lambda = \sqrt{9-\frac{4}{c^{2}}},
  \end{equation}
  but only finitely many for any larger value of $\lambda$.

  (ii) Conversely, suppose $x'$ is an irrational number such that only
  finitely many fractions~$\frac{p}{q}$ satisfy~\eqref{eq:lambda} for
  some $\lambda<3$. Then there exists a unique sorted Markov triple
  $(a,b,c)$ such that $x'$ is equivalent to $x$ defined
  by equation~\eqref{eq:markov_x}.
\end{theorem}

\begin{remark}
  \label{rem:markov}
  A few remarks, first some terminology. 

  (i) The \emph{Lagrange number}~$L(x)$ of an irrational number $x$ is
  defined by
  \begin{equation*}
    L(x) = \sup\big\{\lambda\in\R\,\big|\,\text{infinitely many fractions
      $\tfrac{p}{q}$ satisfy~\eqref{eq:lambda}}\big\},
  \end{equation*}
  and the set of Lagrange numbers $\{L(x)\,|\,x\in\R\setminus\Q\}$ is
  called the \emph{Lagrange
    spectrum}. Equation~\eqref{eq:markov_lambda} describes the part of
  the Lagrange spectrum below $3$, and equation~\eqref{eq:markov_x}
  provides representatives of the corresponding equivalence classes of
  irrational numbers.

  (ii) It may seem strangely unsymmetric that $p_{2}$ appears in
  equation~\eqref{eq:markov_x} and $p_{1}$ does not. The appearance is
  deceptive: Markov's equation~\eqref{eq:markov} and
  equation~\eqref{eq:p1p2} imply that equation~\eqref{eq:markov_x} is
  equivalent to
  \begin{equation*}
    x = \frac{p_{1}}{b}-\frac{a}{bc}+\frac{3}{2} 
    + \sqrt{\frac{9}{4}-\frac{1}{c^{2}}}\,.
  \end{equation*}

  (iii) The three integers of a Markov triple are pairwise coprime. (This is
  true for $(1,1,1)$, and if it is true for some Markov triple, then
  also for its neighbors.) Therefore, integers $p_{1}$, $p_{2}$
  satisfying~\eqref{eq:p1p2} always exist. Different solutions
  $(p_{1}, p_{2})$ for the same Markov triple lead to equivalent
  values of $x$, differing by integers.

  (iv) The following question is more subtle: Under what conditions
  do \emph{different} Markov triples $(a,b,c)$ and $(a',b',c')$ lead
  to equivalent numbers $x$, $x'$? Clearly, if $c\not=c'$, then $x$
  and $x'$ are not equivalent because $\lambda\not=\lambda'$. But
  Markov triples $(a,b,c)$ and $(b,a,c)$ lead to equivalent
  numbers. In general, the numbers $x$ obtained by~\eqref{eq:markov_x}
  from Markov triples $(a,b,c)$ and $(a',b',c')$ are equivalent if and
  only if one can get from $(a,b,c)$ to $(a',b',c')$ or $(b',a',c')$
  by a finite composition of the involutions $\sigma_{1}$ and
  $\sigma_{2}$ fixing $c$. In this case, let us consider the Markov
  triples \emph{equivalent}. Every equivalence class of Markov triples
  contains exactly one sorted Markov triple. It is not known whether
  there exists only one sorted Markov triple $(a,b,c)$ for every
  Markov number $c$. This was remarked by Frobenius~\cite{frobenius13}
  some one hundred years ago, and the question is still open. The
  affirmative statement is known as the \emph{Uniqueness Conjecture
    for Markov Numbers}. Consequently, it is not known whether there
  is only one equivalence class of numbers $x$ for every Lagrange
  number $L(x)<3$.

  (v) The attribution of Hurwitz's theorem may seem strange. It
  covers only the simplest part of Markov's theorem, and Markov's work
  precedes Hurwitz's. However, Markov's original theorem dealt with
  indefinite quadratic forms (see the following section). Despite its
  fundamental importance, Markov's groundbreaking work gained
  recognition only very slowly. Hurwitz began translating Markov's
  ideas to the setting of Diophantine approximation. As this circle of
  results became better understood by more mathematicians, the
  translation seemed more and more straightforward. Today, both
  versions of Markov's theorem, the Diophantine approximation version
  and the quadratic forms version, are unanimously attributed to
  Markov.
\end{remark}

\section{Markov's theorem on indefinite quadratic forms}
\label{sec:forms}

In this section, we recall the quadratic forms version of Markov's
theorem.

We consider binary quadratic forms 
\begin{equation}
  \label{eq:f}
  f(p,q)=Ap^{2}+2Bpq+Cq^{2},
\end{equation}
with real coefficients $A$, $B$, $C$. The \emph{determinant} of such a
form is the determinant of the corresponding symmetric $2\times
2$-matrix,
\begin{equation}
  \label{eq:detf}
  \det f=AC-B^{2}.
\end{equation}
Markov's theorem deals with indefinite forms, i.e., forms with
\begin{equation*}
  \det f < 0.
\end{equation*}
In this case, the quadratic polynomial 
\begin{equation}
  \label{eq:poly}
  f(x,1)=Ax^{2}+2Bx+C
\end{equation}
has two distinct real roots, 
\begin{equation}
  \label{eq:zeros}
  \frac{-B\pm\sqrt{-\det f}}{A},
\end{equation}
provided $A\not=0$. If $A=0$, it makes sense to consider
$\frac{-C}{2B}$ and $\infty$ as two roots in the real projective line
$\R P^{1}\cong \R\cup\{\infty\}$. Then the following statements are
equivalent:

\smallskip
\begin{compactenum}[(i)]
\item The polynomial~\eqref{eq:poly} has at least one root in
  $\Q\cup\{\infty\}$.
\item There exist integers $p$ and $q$, not both zero, such that
  $f(p,q)=0$.
\end{compactenum}
\smallskip

Conversely, one may ask: For which indefinite forms $f$ does the set of values
\begin{equation*}
  \big\{f(p,q)\,\big|\,(p,q)\in\Z^{2}, (p,q)\not=(0,0)\big\}\subseteq\R
\end{equation*}
stay farthest away from $0$. This makes sense if we require the forms
$f$ to be normalized to $\det f=-1$. Equivalently, we may ask: For
which forms is the infimum
\begin{equation}
  \label{eq:mf}
  M(f)=\inf_{\substack{(p,q)\in\Z^{2}\\(p,q)\not=0}} 
    \frac{|f(p,q)|}{\sqrt{|\det f|}}
\end{equation}
maximal? These forms are ``most unlike'' forms with at least one
rational root, for which $M(f)=0$. Korkin and
Zolotarev~\cite{korkine73} gave the following answer:

\begin{theorem}[Korkin~\& Zolotarev]
  Let $f$ be an indefinite binary quadratic form with real
  coefficients. If $f$ is equivalent to the form 
  \begin{equation*}
    p^{2}-pq-q^{2},
  \end{equation*}
  then 
  \begin{equation*}
    M(f) = \frac{2}{\sqrt{5}}\,. 
  \end{equation*}
  Otherwise,
  \begin{equation}
    \label{eq:korkin_otherwise}
    M(f)\leq\frac{1}{\sqrt{2}}\,.
  \end{equation}
\end{theorem}

Binary quadratic forms $f$, $\tilde{f}$ are called \emph{equivalent} if 
there are integers $a$, $b$, $c$, $d$ satisfying
\begin{equation*}
  |ad-bc|=1,
\end{equation*}
such that 
\begin{equation}
  \label{eq:equiv_form}
  \tilde{f}(p,q)=f(ap+bq,cp+dq).
\end{equation}
Equivalent quadratic forms attain the same values. 

Hurwitz's theorem is roughly the Diophantine approximation version of
\mbox{Korkin} \& Zolotarev's theorem. They did not publish a proof, but
Markov obtained one from them personally. This was the starting point
of his work on quadratic forms~\cite{markoff79, markoff80}, which
establishes a bijection between the classes of forms for which
$M(f)\geq\frac{2}{3}$ and sorted Markov triples:

\begin{theorem}[Markov, quadratic forms version]
 \label{thm:markov}
 (i) Let $(a,b,c)$ be any Markov triple, let $p_{1}$, $p_{2}$ be
 integers satisfying equation~\eqref{eq:p1p2}, let 
 \begin{equation}
   \label{eq:markov_x0}
   x_{0}=\frac{p_{2}}{a}+\frac{b}{ac}-\frac{3}{2}\,,
 \end{equation}
 let 
 \begin{equation}
   \label{eq:markov_r}
   r = \sqrt{\frac{9}{4}-\frac{1}{c^{2}}}\,
 \end{equation}
 and let $f$ be the indefinite quadratic form
 \begin{equation}
   \label{eq:markov_f}
   f(p,q)=p^{2}-2x_{0}\,pq+(x_{0}^{2}-r^{2})\,q^{2}.
 \end{equation}
 Then 
 \begin{equation}
   \label{eq:markov_mf}
   M(f)=\frac{1}{r},
 \end{equation}
 and the infimum in \eqref{eq:mf} is attained.

 (ii) Conversely, suppose $\tilde{f}$ is an indefinite binary quadratic form with
 \begin{equation*}
   M(\tilde{f}) > \frac{2}{3}.
 \end{equation*}
 Then there is a unique sorted Markov triple $(a,b,c)$ such that
 $\tilde{f}$ is equivalent to a multiple of the form $f$ defined by
 equation~\eqref{eq:markov_f}.
\end{theorem}

Note that the number $x$ defined by~\eqref{eq:markov_x} is a root of
the form $f$ defined by~\eqref{eq:markov_f}, and
$M(f)=\frac{2}{L(x)}$.  Table~\ref{tab:forms} lists representatives
$f(p,q)$ of the five classes of forms with the largest values of
$M(f)$.
\begin{remark}
  Here, too, the apparent asymmetry between $p_{1}$ and $p_{2}$ is
  deceptive (cf.\
  Remark~\ref{rem:markov}~(ii)). Equation~\eqref{eq:markov_x0} is
  equivalent to
  \begin{equation*}
    x_{0}=\frac{p_{1}}{b}-\frac{a}{bc}+\frac{3}{2}\,.
  \end{equation*}
\end{remark}

\begin{table}[h]
  \centering
  \renewcommand{\arraystretch}{1.5}
  \begin{tabular}{c|
                  >{$}c<{$}|
                  >{$}r<{$}@{$\;=\;$}>{$}l<{$}|
                  >{$}c<{$}@{ }>{$}c<{$}@{ }>{$}c<{$}|
                  >{$}c<{$}@{ }>{$}c<{$}}
    rank & f(p,q) & \multicolumn{2}{c|}{$M(f)$} & \ a\  & \ b\  & \ c\  & \ p_{1}\  &
                                                                       \ p_{2}\ \\
    \hline
    1
    & p^2-pq-q^2
    & \frac{2}{\sqrt{5}} & 0.89\ldots   
    & 1 & 1 & 1
    & 0 & 1
    \\
    2
    & p^2-2q^2
    & \frac{1}{\sqrt{2}} & 0.70\ldots 
    & 1 & 1 & 2
    & -1 & 1
    \\
    3
    &5p^2+pq-11q^2
    & \frac{10}{\sqrt{221}} & 0.67\ldots
    & 1 & 2 & 5
    & -1 & 2
    \\
    4
    & 13p^2+23pq-19q^2
    & \frac{26}{\sqrt{1517}} & 0.667\ldots
    & 1 & 5 & 13
    & -3 & 2
    \\
    5
    & 29p^2-5pq-65q^2
    & \frac{58}{\sqrt{7565}} & 0.6668\ldots
    & 2 & 5 & 29
    & -7 & 3
  \end{tabular}
  \caption{The five classes of indefinite quadratic forms whose values
    stay farthest away from zero}
\label{tab:forms}
\end{table}

\section{The hyperbolic plane}
\label{sec:hyperbolic}

We use the half-space model of the hyperbolic plane for all
calculations. In this section, we summarize some basic facts.

The hyperbolic plane is represented by the upper half-plane of the
complex plane,
\begin{equation*}
  H^{2}=\{z\in\C\,|\,\im z > 0\},
\end{equation*}
where the length of a curve $\gamma:[t_{0},t_{1}]\rightarrow H^{2}$ is
defined as
\begin{equation*}
  \int_{t_{0}}^{t_{1}}\frac{|\dot{\gamma}(t)|}{\im \gamma(t)}\,dt.
\end{equation*}
The model is conformal, i.e., hyperbolic angles are equal to euclidean
angles. The group of isometries is the projective general linear group,
\begin{equation*}
  \begin{split}
    \PGLTR &= \GLTR/\R^{*} \\
           & \cong \big\{A\in\GLTR\,\big|\,|\det
    A|=1\big\}/\{\pm\Id\},
  \end{split}
\end{equation*}
where the action $M:\PGLTR\rightarrow\Isom(H^{2})$ is
defined as follows: 

For
\begin{equation*}
  A=    
  \big(
  \begin{smallmatrix}
    a & b \\ c & d
  \end{smallmatrix}
  \big)\in\GLTR,
\end{equation*}
\begin{equation*}
  M_{A}(z)=
  \begin{dcases*}
    \frac{az+b}{cz+d} & if $\det A>0$,\\
    \frac{a\bar{z}+b}{c\bar{z}+d} & if $\det A<0$.
  \end{dcases*}
\end{equation*}
The isometry $M_{A}$ preserves orientation if $\det A>0$ and reverses
orientation if $\det A<0$. The subgroup of orientation preserving
isometries is therefore $\PSLTR\cong\SLTR/\{\pm\Id\}$.

Geodesics in the hyperbolic plane are euclidean half circles
orthogonal to the real axis or euclidean vertical lines (see
Fig.~\ref{fig:hplane}).
\begin{figure}
\labellist
\small\hair 2pt
 \pinlabel {$x+iy_{0}$} [l] at 80 131
 \pinlabel {$x+iy_{1}$} [l] at 80 106
 \pinlabel {$\displaystyle\log\frac{y_{0}}{y_{1}}$} [r] <-2pt, 0pt> at 80 118
 \pinlabel {geodesics} [lb] at 103 83
 \pinlabel {horocycles} [r] at 175 76
 \pinlabel {$p'^{2}$} [l] at 272 98
 \pinlabel {$\displaystyle\frac{1}{q^{2}}$} [l] at 272 72
 \pinlabel {$\displaystyle \frac{p}{q}$} [t] at 211 10
 \pinlabel {$h(p',0)$} [b] at 210 98
 \pinlabel {$h(p,q)$} [l] at 242 42
\endlabellist
\centering
\includegraphics[scale=1.0]{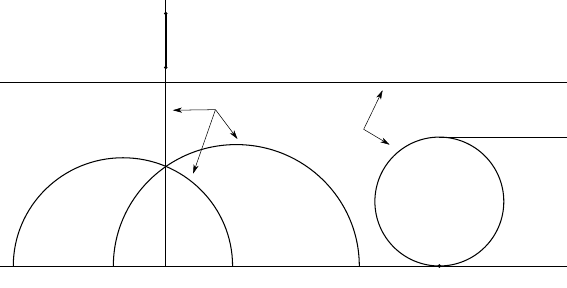}\vspace{0.25\baselineskip}
\caption{Geodesics and horocycles}
\label{fig:hplane}
\end{figure}
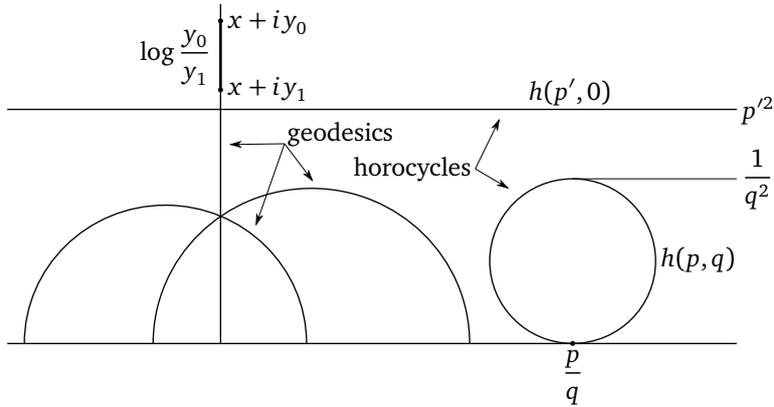
The hyperbolic distance between points $x+iy_{0}$ and
$x+iy_{1}$ on a vertical geodesic is 
\begin{equation*}
  \Big|\log\frac{y_{1}}{y_{0}}\Big|.
\end{equation*}

Apart from geodesics, horocycles will play an important role. They are
the limiting case of circles as the radius tends to
infinity. Equivalently, horocycles are complete curves of curvature
$1$. In the half-space model, horocycles are represented as euclidean
circles that are tangent to the real line, or as horizontal lines. The
center of a horocycle is the point of tangency with the real line, or
$\infty$ for horizontal horocycles.

The points on the real axis and $\infty\in\C P^{1}$ are called ideal
points. They do not belong to the hyperbolic plane, but they
correspond to the ends of geodesics. All horocycles centered at an
ideal point $x\in\R\cup\{\infty\}$ intersect all geodesics ending in
$x$ orthogonally. In the proof of Proposition~\ref{prop:dhv}, we will
use the fact that two horocycles centered at the same ideal point are
equidistant curves.

\section{Dictionary: horocycle --- 2D vector}
\label{sec:horo}

We assign a horocycle $h(p,q)$ to every $(p,q)\in\R^{2}\setminus\{(0,0)\}$
as follows (see Fig.~\ref{fig:hplane}):

\smallskip
\begin{compactitem}
\item For $q\not=0$, let $h(p,q)$ be the horocycle at
  $\frac{p}{q}$ with euclidean diameter~$\frac{1}{q^{2}}$.
\item Let $h(p,0)$ be the horocycle at~$\infty$ at height~$p^{2}$.
\end{compactitem}
\smallskip

The map $(p,q)\mapsto h(p,q)$ from $\R^{2}\setminus\{0\}$ to the space
of horocycles is surjective and two-to-one, mapping $\pm(p,q)$ to the
same horocycle. The map is equivariant with respect to the
$\PGLTR$-action~\cite[p.~665]{fock07}. More precisely:

\begin{proposition}[Equivariance]
  \label{prop:equivariant}
  For $A\in\GLTR$ satisfying $|\det A|=1$ and for
  $v\in\R^{2}\setminus\{0\}$, the hyperbolic isometry $M_{A}$ maps the
  horocycle $h(v)$ to $h(Av)$.
\end{proposition}

\begin{proof}
  This can of course be shown by direct calculation. To simplify the
  calculations, note that every isometry of $H^{2}$ can be represented
  as a composition of isometries of the following types:
  \begin{equation}\label{eq:3moeb}
    z\mapsto z+b,\quad 
    z\mapsto \lambda z,\quad 
    z\mapsto-\bar{z},\quad
    z\mapsto \frac{1}{\bar{z}}
  \end{equation}
  (where $b\in\R$, $\lambda\in\R_{>0}$). The corresponding normalized matrices are
  \begin{equation}
    \label{eq:3mat}
    \begin{pmatrix}
      1 & b \\ 0 & 1
    \end{pmatrix},\quad
    \begin{pmatrix}
      \lambda^{\frac{1}{2}} & 0 \\ 0 & \lambda^{-\frac{1}{2}}
    \end{pmatrix},\quad
    \begin{pmatrix}
      -1 & 0 \\ 0 & 1
    \end{pmatrix},\quad
    \begin{pmatrix}
      0 & 1 \\ 1 & 0
    \end{pmatrix}.
  \end{equation}
  (The first two maps preserve orientation, the other two reverse it.)
  It is therefore enough to do the simpler calculations for these
  maps. (For the inversion, Fig.~\ref{fig:inversion} indicates an
  alternative geometric argument, just for fun.)
  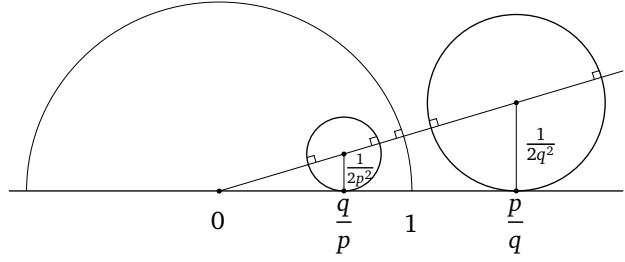
\begin{figure}
    \labellist
    \small\hair 3pt
    \pinlabel {$\displaystyle 0\vphantom{\frac{q}{p}}$} [t] at 78 11
    \pinlabel {$\displaystyle\frac{q}{p}$} [t] at 125 11
    \pinlabel {$\displaystyle 1\vphantom{\frac{q}{p}}$} [t] at 150 10
    \pinlabel {$\displaystyle\frac{p}{q}$} [t] at 189 11
    \pinlabel {\footnotesize$\frac{1}{2p^{2}}$} [l] at 122 19
    \pinlabel {$\frac{1}{2q^{2}}$} [l] at 189 28
    \endlabellist
    \centering
    \includegraphics[scale=1.0]{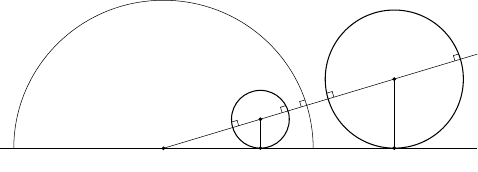}
    \medskip
    \caption{Horocycle $h(p,q)$ and image under inversion
      $z\mapsto\frac{1}{\bar z}$}
    \label{fig:inversion}
  \end{figure}
\end{proof}

\section{Signed distance of two horocycles}
\label{sec:dist_h_h}

The \emph{signed distance $d(h_{1},h_{2})$} of horocycles $h_{1}$,
$h_{2}$ is defined as follows (see Fig.~\ref{fig:dist_h_h}):
\begin{figure}
  \labellist
  \small\hair 3pt 
  \pinlabel {$d>0$} [b] <2pt,1pt> at 83 41 
  \pinlabel {$d<0$} [b] at 218 21
  \endlabellist
  \centering
  \includegraphics[scale=1.0]{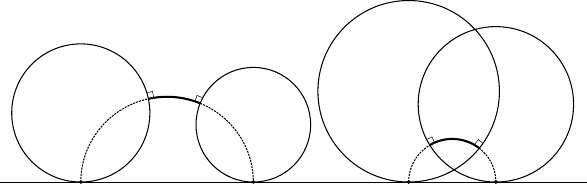}
  \caption{The signed distance of horocycles}
  \label{fig:dist_h_h}
\end{figure}
\begin{compactitem}
\item If $h_{1}$ and $h_{2}$ are centered at different points and do
  not intersect, then~$d(h_{1},h_{2})$ is the length of the geodesic
  segment connecting the horocycles and orthogonal to both. (This is
  just the hyperbolic distance between the horocycles.)
\item If $h_{1}$ and $h_{2}$ do intersect, then $d(h_{1},h_{2})$ is
  the length of that geodesic segment, taken negative. (If $h_{1}$
  and $h_{2}$ are tangent, then $d(h_{1},h_{2})=0$.)
\item If $h_{1}$ and $h_{2}$ have the same center, then
  $d(h_{1},h_{2})=-\infty$.
\end{compactitem}

\begin{remark}
  If horocycles $h_{1}$, $h_{2}$ have the same center, they are
  equidistant curves with a well defined finite distance. But their
  signed distance is defined to be $-\infty$. Otherwise, the map
  $(h_{1},h_{2})\mapsto d(h_{1},h_{2})$ would not be continuous on the
  diagonal.
\end{remark}

\begin{proposition}[Signed distance of horocycles]
  \label{prop:dhh}
  The signed distance of two horocycles $h_{1}=h(p_{1},q_{1})$ and
  $h_{2}=h(p_{2},q_{2})$ is
  \begin{equation}
    \label{eq:dhh}
    d(h_{1},h_{2})=2\log|p_{1}q_{2}-p_{2}q_{1}|.
  \end{equation}
\end{proposition}

\begin{proof}
  It is easy to derive equation~\eqref{eq:dhh} if one horocycle is
  centered at $\infty$ (see Fig.~\ref{fig:hplane}). To prove the
  general case, apply the hyperbolic isometry
  \begin{equation*}
    M_{A}(z)=\frac{1}{z-\frac{p_{1}}{q_{1}}},
    \qquad
    A= 
    \begin{pmatrix}
      0 & 1 \\
      1 & -\frac{p_{1}}{q_{1}}
    \end{pmatrix}
  \end{equation*}
  that maps one horocycle center to $\infty$ and use
  Proposition~\ref{prop:equivariant}.
\end{proof}

\section{Ford circles and Farey tessellation}
\label{sec:farey}

Figure~\ref{fig:integer_cycles}
\begin{figure}
  \centering
  \labellist
  \footnotesize\hair 3pt
  \pinlabel {$-1$} [t] at 16 5
  \pinlabel {$0$} [t] at 100 5
  \pinlabel {$1$} [t] at 183 5
  \pinlabel {$2$} [t] at 267 5
  \pinlabel {$\frac{1}{2}$} [t] at 142 5
  \pinlabel {$\frac{1}{3}$} [t] at 128 5
  \pinlabel {$\frac{2}{3}$} [t] at 156 5
  \pinlabel {$\frac{1}{4}$} [t] at 121 5
  \pinlabel {$\frac{3}{4}$} [t] at 162.5 5
  \pinlabel {$\frac{1}{5}$} [t] at 116.75 5
  \pinlabel {$\frac{2}{5}$} [t] at 133.5 5
  \pinlabel {$\frac{3}{5}$} [t] at 150 5
  \pinlabel {$\frac{4}{5}$} [t] at 167 5
  \endlabellist
  \includegraphics{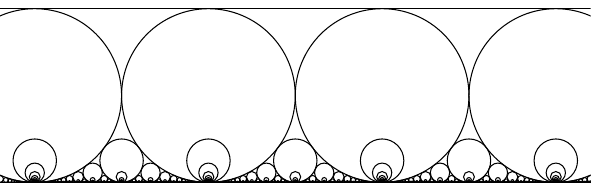}
  \vspace{5pt}% space needed for labels
  \caption{Horocycles $h(p,q)$ with integer parameters $(p,q)\in\Z^{2}$}
  \label{fig:integer_cycles}
  \bigskip
  \labellist
  \footnotesize\hair 3pt
  \pinlabel {$-1$} [t] at 16 5
  \pinlabel {$0$} [t] at 100 5
  \pinlabel {$1$} [t] at 183 5
  \pinlabel {$2$} [t] at 267 5
  \pinlabel {$\frac{1}{2}$} [t] at 142 5
  \pinlabel {$\frac{1}{3}$} [t] at 128 5
  \pinlabel {$\frac{2}{3}$} [t] at 156 5
  \pinlabel {$\frac{1}{4}$} [t] at 121 5
  \pinlabel {$\frac{3}{4}$} [t] at 162.5 5
  \pinlabel {$\frac{1}{5}$} [t] at 116.75 5
  \pinlabel {$\frac{2}{5}$} [t] at 133.5 5
  \pinlabel {$\frac{3}{5}$} [t] at 150 5
  \pinlabel {$\frac{4}{5}$} [t] at 167 5
  \endlabellist
  \includegraphics{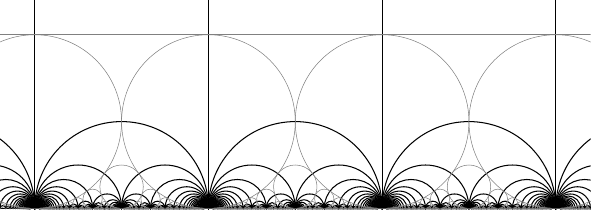}
    \vspace{5pt}% space needed for labels
  \caption{Ford circles and Farey tessellation}
  \label{fig:farey}
\end{figure}
shows the horocycles $h(p,q)$ with integer parameters
$(p,q)\in\Z^{2}$. There is an infinite family of such integer
horocycles centered at each rational number and at $\infty$. (Only
the lowest horocycle centered at $\infty$ is shown to save space.)
Integer horocycles $h(p_{1},q_{1})$ and $h(p_{2},q_{2})$ with
different centers $\frac{p_{1}}{q_{1}}\not=\frac{p_{2}}{q_{2}}$ do not
intersect. This follows from Proposition~\ref{prop:dhh}, because
$p_{1}q_{2}-p_{2}q_{1}$ is a non-zero integer. They touch if and only
if $p_{1}q_{2}-p_{2}q_{1}=\pm 1$. This can happen only if both
$(p_{1},q_{1})$ and $(p_{2},q_{2})$ are coprime, that is,
if~$\frac{p_{1}}{q_{1}}$ and~$\frac{p_{2}}{q_{2}}$ are reduced
fractions representing the respective horocycle centers.

Figure~\ref{fig:farey} shows the horocycles $h(p,q)$ with integer and
coprime parameters~$(p,q)$. They are called \emph{Ford circles}. There
is exactly one Ford circle centered at each rational number and at
$\infty$. If one connects the ideal centers of tangent Ford circles
with geodesics, one obtains the \emph{Farey tessellation}, which is
also shown in the figure. The Farey tessellation is an ideal
triangulation of the hyperbolic plane with vertex set
$\Q\cup\{\infty\}$. (A thorough treatment can be found
in~\cite{bonahon09}.)

We will see that Markov triples correspond to ideal triangulations of
the hyperbolic plane (as universal cover of the modular torus), and
$(1,1,1)$ corresponds to the Farey tessellation
(Sec.~\ref{sec:ideal_triang}). The Farey tessellation also comes up
when one considers the minima of \emph{definite} quadratic forms
(Sec.~\ref{sec:point}).

\section{Signed distance of a horocycle and a geodesic}
\label{sec:dhg}

For a horocycle $h$ and a geodesic $g$, the \emph{signed distance}
$d(h,g)$ is defined as follows (see Fig.~\ref{fig:dist_h_g}):
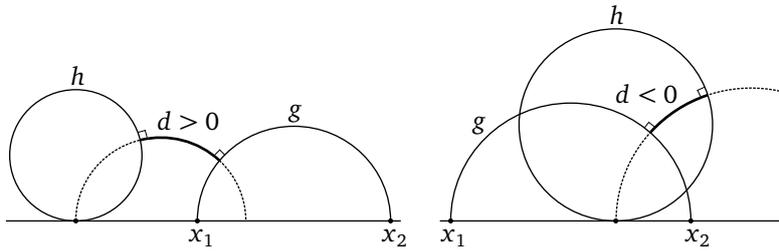
\begin{figure}
  \labellist
  \small\hair 2pt
  \pinlabel {$h$} [b] at 27 50
  \pinlabel {$g$} [b] at 108 36
  \pinlabel {$h$} [b] at 228 73
  \pinlabel {$g$} [br] at 181 32
  \pinlabel {$d>0$} [b] <4pt,2pt> at 64 31
  \pinlabel {$d<0$} [br] <1pt,2pt> at 251 42
  \pinlabel {$x_{1}$} [t] <2pt, 0pt> at 71 0
  \pinlabel {$x_{2}$} [t] <2pt, 0pt> at 143 0
  \pinlabel {$x_{1}$} [t] <2pt, 0pt> at 165 0
  \pinlabel {$x_{2}$} [t] <2pt, 0pt> at 256 0
  \endlabellist
  \centering
  \vspace{10pt}
  \includegraphics[scale=1.0]{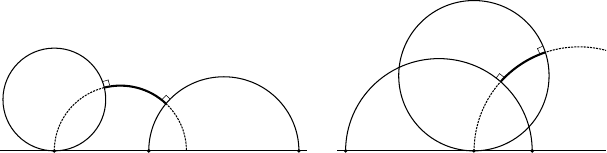}
  \caption{The signed distance $d=d(h,g)$ of a horocycle $h$ and a geodesic~$g$}
  \label{fig:dist_h_g}
\end{figure}

\smallskip
\begin{compactitem}
\item If $h$ and $g$ do not intersect, then~$d(h,g)$ is the length of
  the geodesic segment connecting $h$ and $g$ and orthogonal to
  both. (This is just the hyperbolic distance between $h$ and $g$.)
\item If $h$ and $g$
  do intersect, then $d(h,g)$ is the length of that
  geodesic segment, taken negative.
\item If $h$ and $g$ are tangent then $d(h,g)=0$. 
\item If $g$ ends in the center of $h$ then $d(h,g)=-\infty$.
\end{compactitem}
\smallskip

An equation for the signed distance to a vertical geodesic is
particularly easy to derive:

\begin{proposition}[Signed distance to a vertical geodesic]
  \label{prop:dhv}
  Consider a horocycle $h=h(p,q)$ with $q\not=0$ and a vertical
  geodesic $g$ from~$x\in\R$ to~$\infty$. Their signed distance is
  \begin{equation}
    \label{eq:dhv}
    d(h,g)=\log\Big(2q^{2}\Big|x-\frac{p}{q}\Big|\Big).
  \end{equation}
\end{proposition}

\begin{proof}
  See Fig.~\ref{fig:vertical}. 
\end{proof}

\begin{figure}
  \labellist
  \small\hair 3pt 
  \pinlabel {$x$} [t] at 73 10 
  \pinlabel {$\displaystyle\frac{p}{q}$} [t] at 126 10 
  \pinlabel {$d$} [b] at 85 60 
  \pinlabel {$d$} [l] at 126 93 
  \pinlabel {$\displaystyle \frac{1}{q^{2}}$} [l] at 186 71 
  \pinlabel {$2\big|x-\frac{p}{q}\big|$} [l] at 186 115
  \pinlabel {$g$} [r] at 74 119
  \pinlabel {$h$} [tr] at 152 58
  \endlabellist
  \centering
  \includegraphics[width=150pt]{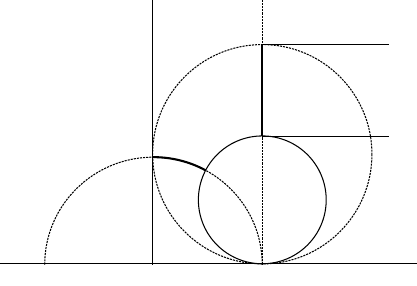}%
  \medskip
  \caption{Signed distance of horocycle $h=h(p,q)$ and vertical geodesic~$g$}
  \label{fig:vertical}
\end{figure}
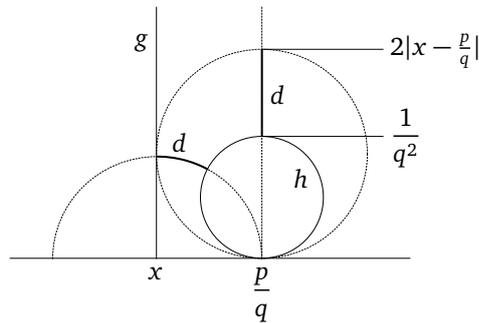

Equation~\eqref{eq:dhv} suggests a geometric interpretation of
Hurwitz's theorem and the Diophantine approximation version of
Markov's theorem: A fraction $\frac{p}{q}$ satisfies
inequality~\eqref{eq:lambda} if and only if
\begin{equation}
  d\big(h(p,q),g\big)<-\log\frac{\lambda}{2}\,.
\end{equation}
The following section contains a proof of Hurwitz's theorem based on
this observation. An equation for the signed distance to a general
geodesic will be presented in Proposition~\ref{prop:dhgf}.

\section{Proof of Hurwitz's theorem}
\label{sec:hurwitz}

Let $x$ be an irrational number and let $g$ be the vertical geodesic
from $x$ to $\infty$. By Proposition~\ref{prop:dhv}, part (i) of
Hurwitz's theorem is equivalent to the statement: 

Infinitely many Ford circles $h$ satisfy
\begin{equation}
  \label{eq:hurwitz_geo}
  d(h,g)<-\log\frac{\sqrt{5}}{2}\,.
\end{equation}

This follows from the following lemma. Let us say that the
\emph{midpoint} of an edge of the Farey tessellation is the point
where the horocycles centered at its ends meet (see
Fig.~\ref{fig:farey}). Accordingly, we say that a geodesic
\emph{bisects} an edge of the Farey tessellation if it passes through
the midpoint of the edge (see Fig.~\ref{fig:furthestgeo1}).

\begin{lemma}
  \label{lem:hurwitz_farthest}
  Suppose a geodesic $g$ crosses an ideal triangle $T$ of the Farey
  tessellation. If $g$ is one of the three geodesics bisecting two
  sides of $T$, then
  \begin{equation*}
    d(h,g)=-\log\frac{\sqrt{5}}{2}
  \end{equation*}
  for all three Ford circles $h$ at the vertices of $T$.  Otherwise,
  inequality~\eqref{eq:hurwitz_geo} holds for at least one of these
  three Ford circles.
\end{lemma}

\begin{proof}[Proof of Lemma~\ref{lem:hurwitz_farthest}]
  This is the simplest case of Propositions~\ref{prop:crossing1} and~\ref{prop:crossing2}, and
  easy to prove independently. Note that it is enough to
  consider the ideal triangle $0$, $1$, $\infty$, and geodesics
  intersecting its two vertical sides (see
  Fig.~\ref{fig:furthestgeo1}).
\end{proof}

\begin{figure}
  \labellist
  \small\hair 2pt 
  \pinlabel {$\frac{1}{2}-\frac{\sqrt{5}}{2}$} at 5 0 
  \pinlabel {$0$} at 43 0 
  \pinlabel {$\frac{1}{2}$} at 76 0
  \pinlabel {$1$} at 105 0 
  \pinlabel {$\frac{1}{2}+\frac{\sqrt{5}}{2}=\Phi$} at 152 0 
  \pinlabel {$1$} [l] at 201 72 
  \pinlabel {$\frac{\sqrt{5}}{2}$} [l] at 203 81 
  \pinlabel {\scriptsize $d$} [r] at 73 76
  \pinlabel {$g_{1}$} [r] at 6 25
  \pinlabel {$g$} [l] at 142 93
  \pinlabel {$\frac{\sqrt{5}}{2}$} [l] at 53 52
  \endlabellist
  \centering
  \includegraphics[scale=1.0]{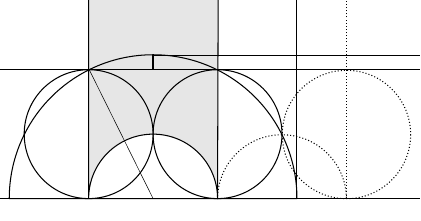}
  \caption{Geodesic $g_{1}$ bisecting the two vertical sides of the
    triangle $0,1,\infty$, and geodesic $g$ from $\Phi$ to $\infty$}
  \label{fig:furthestgeo1}
\end{figure}

To deduce part (i) of Hurwitz's theorem, note that since $x$ is
irrational, the geodesic $g$ from $x$ to $\infty$ passes through
infinitely many triangles of the Farey tessellation. For each of these
triangles, at least one of its Ford circles $h$
satisfies~\eqref{eq:hurwitz_geo}, by
Lemma~\ref{lem:hurwitz_farthest}. (The geodesic $g$ does not bisect
two sides of any Farey triangle. Otherwise, $g$ would bisect two sides
of all Farey triangles it enters; see Fig.~\ref{fig:furthestgeo1},
where the next triangle is shown with dashed lines. This contradicts
$g$ ending in the vertex $\infty$ of the Farey tessellation.) 

For consecutive triangles that $g$ crosses, the same horocycle may
satisfy~\eqref{eq:hurwitz_geo}. But this can happen only finitely many
times (otherwise $x$ would be rational), and then the geodesic will
never again intersect a triangle incident with this horocycle. Hence,
infinitely many Ford circles satisfy~\eqref{eq:hurwitz_geo}, and this
completes the proof of part (i).

To prove part (ii) of Hurwitz's theorem, we have to show that for
\begin{equation*}
  x=\Phi
  \quad\text{and}\quad 
  \epsilon>0,
\end{equation*}
only finitely many Ford circles $h$ satisfy 
\begin{equation}
  \label{eq:d_less_delta}
  d(h,g)<-\log\frac{\sqrt{5}}{2}-\epsilon,
\end{equation}
where $g$ is the geodesic from $\Phi$ to $\infty$. 

To this end, let $g_{1}$ be the geodesic from
$\Phi=\frac{1}{2}(1+\sqrt{5})$ to $\frac{1}{2}(1-\sqrt{5})$, see
Fig.~\ref{fig:furthestgeo1}. For every Ford circle $h$,
\begin{equation*}
  d(h,g_{1})\geq -\log\frac{\sqrt{5}}{2}. 
\end{equation*}
Indeed, the distance is equal to $-\log\frac{\sqrt{5}}{2}$ for all Ford
circles that $g_{1}$ intersects, and positive for all others.

Because the geodesics $g$ and $g_{1}$ converge at the common end
$\Phi$, there is a point $P\in g$ such that all Ford circles $h$
intersecting the ray from $P$ to $\Phi$ satisfy
\begin{equation*}
  |d(g,\Phi)-d(g_{1},\Phi)|<\epsilon,
\end{equation*}
and hence 
\begin{equation*}
  d(g,\Phi)\geq-\log\frac{\sqrt{5}}{2}-\epsilon.
\end{equation*}
On the other hand, the complementary ray of $g$, from $P$ to $\infty$,
intersects only finitely many Ford circles. Hence, only finitely many
Ford circles satisfy~\eqref{eq:d_less_delta}, and this completes the
proof of part (ii).

\begin{remark}
  \label{rem:vahlen}
  The gist of the above proof is deducing Hurwitz's theorem from the
  fact that the geodesic $g$ from an irrational number $x$ to $\infty$
  crosses infinitely many Farey triangles. A weaker statement follows
  from the observation that $g$ crosses infinitely many edges. Since
  each edge has two touching Ford circles at the ends, a crossing
  geodesic intersects at least one of them. Hence there are infinitely
  many fractions satisfying~\eqref{eq:lambda} with $\lambda=2$. In
  fact, at least one of any two consecutive continued fraction
  approximants satisfies this bound. This result is due to
  Vahlen~\cite[p.~41]{perron54} \cite{vahlen95}. The converse is due
  to Legendre~\cite{legendre30} and 65 years older: If a fraction
  satisfies~\eqref{eq:lambda} with $\lambda=2$, then it is a continued
  fraction approximant. A geometric proof using Ford circles is
  mentioned by Speiser~\cite{speiser23} (see Sec.~\ref{sec:intro}).
\end{remark}

\section{Dictionary: geodesic --- indefinite form}
\label{sec:geodesics}

We assign a geodesic $g(f)$ to every indefinite binary quadratic form
$f$ with real coefficients as follows: To the form $f$ with real
coefficients $A$, $B$, $C$ as in~\eqref{eq:f}, we assign the geodesic
$g(f)$ that connects the zeros of the polynomial~\eqref{eq:poly}. (If
$A=0$, one of the zeros is $\infty$, and $g(f)$ is a vertical
geodesic.) The map $f\mapsto g(f)$ from the space of indefinite forms
to the space of geodesics is
\begin{compactitem}
\item surjective and many-to-one: 
$g(f)=g(\tilde{f})\Leftrightarrow\tilde{f}=\mu f$ for some $\mu\in\R^{*}$.
\item equivariant with respect to the left $\GLTR$-actions: 
  \begin{equation*}
    \begin{tikzcd}[column sep=large]
      f \arrow[r, mapsto, "A"]\arrow[d, mapsto, "g"]
% How to put a symbol in the middle of a commutative diagram:
% http://tex.stackexchange.com/questions/119543/how-can-i-get-symbols-to-appear-in-the-middle-of-commutative-diagrams-using-tikz
      \arrow[dr, start anchor=center,end
      anchor=center,draw=none,description]{dr}[description]{A\in\GLTR}
      & f\circ A^{-1} \arrow[d, mapsto, "g"] \\
      g(f) \arrow[r, mapsto, "M_{A}"]
      & M_{A}g(f)=g(f\circ A^{-1})
    \end{tikzcd}
  \end{equation*}
\end{compactitem}

\begin{proposition}
  \label{prop:dhgf}
  The signed distance of the horocycle $h(p,q)$ and the geodesic~$g(f)$ is
  \begin{equation}
    \label{eq:dhgf}
    d\big(h(p,q),g(f)\big)=\log\frac{|f(p,q)|}{\sqrt{-\det f}}\,.
  \end{equation}
\end{proposition}

\begin{proof}
  First, consider the case of horizontal horocycles ($q=0$). If $g(f)$
  is a vertical geodesic ($f(p,0)=0$), equation~\eqref{eq:dhgf} is
  immediate. Otherwise, note that
  ${p^{2}\sqrt{-\det f}}/{|f(p,0)|}$ is half the distance between
  the zeros~\eqref{eq:zeros}, hence the height of the geodesic.

  The general case reduces to this one: For any $A\in\GLTR$ with
  $|\det A|=1$ and $A\big(
  \begin{smallmatrix}
    p \\ q
  \end{smallmatrix}
  \big)
  =
  \big(
  \begin{smallmatrix}
    \tilde{p} \\ 0
  \end{smallmatrix}
  \big)
  $,
  \begin{equation*}
    \begin{split}
      d\big(h(p,q),g(f)\big)
      &=d\big(M_{A}h(p,q),M_{A}g(f)\big)
      =d\big(h(\tilde{p},0),g(f\circ A^{-1})\big)\\
      &=\log\frac{|(f\circ A^{-1})(\tilde{p},0)|}{\sqrt{-\det(f\circ
          A^{-1})}}
      =\log\frac{|f(p,q)|}{\sqrt{-\det f}}\,.\qedhere
    \end{split}
  \end{equation*}
  \end{proof}

Equation~\eqref{eq:dhgf} suggests a geometric interpretation of the
quadratic forms version of Markov's theorem, and it is easy to prove
most of Korkin~\& Zolotarev's theorem (just replace
inequality~\eqref{eq:korkin_otherwise} with $M(f)<\frac{2}{\sqrt{5}}$)
by adapting the proof of Hurwitz's theorem in Sec.~\ref{sec:hurwitz}.
To obtain the complete Markov theorem, more hyperbolic geometry is
needed. This this is the subject of the following sections.

\section{Decorated ideal triangles}
\label{sec:ideal_triang}

In this and the following section, we review some basic facts from
Penner's theory of decorated Teichm{\"u}ller
spaces~\cite{penner87,penner12}. The material of this section, up to
and including equation~\eqref{eq:c} is enough to treat crossing
geodesics in Sec.~\ref{sec:crossing}. Ptolemy's relation is needed for
the geometric interpretation of Markov's equation in
Sec.~\ref{sec:triangulations}.

An \emph{ideal triangle} is a closed region in the hyperbolic plane
that is bounded by three geodesics (the \emph{sides}) connecting three
ideal points (the \emph{vertices}). Ideal triangles have dihedral
symmetry, and any two ideal triangles are isometric. That is, for any
pair of ideal triangles and any bijection between their vertices,
there is a unique hyperbolic isometry that maps one to the other and
respects the vertex matching. A \emph{decorated ideal triangle} is an
ideal triangle together with a horocycle at each vertex
(Fig.~\ref{fig:ideal_triang}).
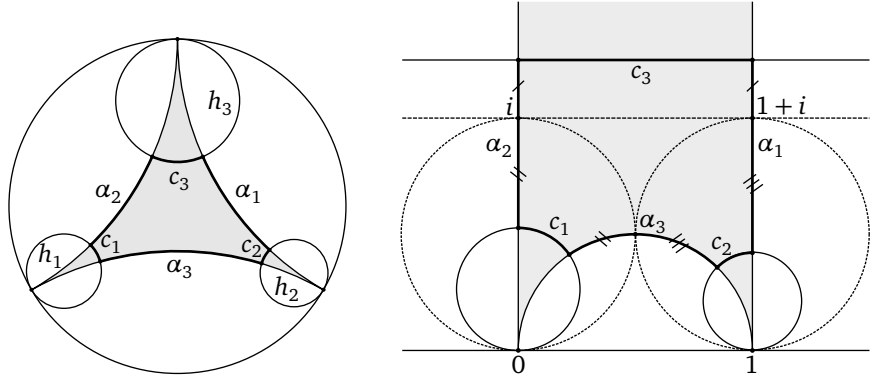
\begin{figure}
  \labellist
  \small\hair 2pt
  \pinlabel {$\alpha_{3}$} [t] at 64 45
  \pinlabel {$\alpha_{1}$} [bl] at 83 63
  \pinlabel {$\alpha_{2}$} [br] at 44 63
  \pinlabel {$c_{3}$} [t] at 64 78
  \pinlabel {$c_{1}$} [bl] at 33 44
  \pinlabel {$c_{2}$} [br] <-0.5pt, -1.5pt> at 97 44
  \pinlabel {$h_{1}$} [t] <1pt,0pt> at 15 51
  \pinlabel {$h_{2}$} [bl] <0pt, -2pt> at 98 28
  \pinlabel {$h_{3}$} [r] at 86 102
  \endlabellist
  \hfill
  \includegraphics{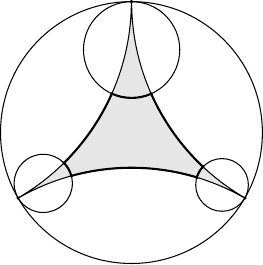}
  \labellist
  \small\hair 2pt
  \pinlabel {$\alpha_{3}$} [bl] <-0.5pt,-0.5pt> at 88 52
  \pinlabel {$\alpha_{1}$} [l] at 131 85
  \pinlabel {$\alpha_{2}$} [r] at 44 85
  \pinlabel {$c_{3}$} [t] at 90 118
  \pinlabel {$c_{1}$} [bl] <-2pt,-1pt> at 56 52
  \pinlabel {$c_{2}$} [br] <1.5pt, 1.5pt> at 124 42
  \pinlabel {$0$} [t] at 44 9
  \pinlabel {$1$} [t] at 131 9
  \pinlabel {$i$} [br] at 44 96
  \pinlabel {$1+i$} [bl] at 131 96
  \endlabellist
  \hfill
  \includegraphics{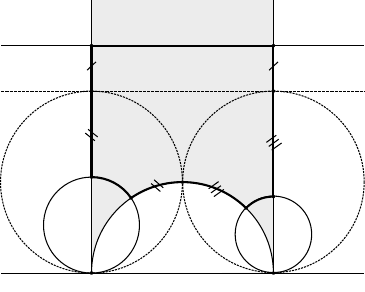}
  \hspace*{\fill}
  \caption{Decorated ideal triangle in the Poincar\'e disk model (left) and
    in the half-plane model (right)}
  \label{fig:ideal_triang}
\end{figure}

Consider a geodesic decorated with two horocycles $h_{1}$, $h_{2}$ at
its ends (for example, a side of an ideal triangle). Let the
\emph{truncated length} of the decorated geodesic be defined as the
signed distance of the horocycles (Sec.~\ref{sec:dist_h_h}),
\begin{equation*}
  \alpha = d(h_{1},h_{2}),
\end{equation*}
and let its \emph{weight} be defined as
\begin{equation*}
  a=e^{\alpha/2}.
\end{equation*}
(We will often use Greek letters for truncated lengths and Latin
letters for weights. The weights are usually called
\emph{$\lambda$-lengths}.)

Any triple $(\alpha_{1},\alpha_{2},\alpha_{3})\in\R^{3}$ of
truncated lengths, or, equivalently, any triple
$(a_{1},a_{2},a_{3})\in\R_{>0}^{3}$ of weights, determines a unique
decorated ideal triangle up to isometry.

Consider a decorated ideal triangle with truncated lengths
$\alpha_{k}$ and weights~$a_{k}$. Its horocycles intersect the
triangle in three finite arcs. Denote their hyperbolic lengths by
$c_{k}$ (see Fig.~\ref{fig:ideal_triang}). The truncated side lengths
determine the horocyclic arc lengths, and vice versa, via the relation
\begin{equation}
  \label{eq:c}
  c_{k}=\frac{a_{k}}{a_{i}a_{j}}=e^{\frac{1}{2}(-\alpha_{i}-\alpha_{j}+\alpha_{k})},
\end{equation}
where $(i,j,k)$ is a permutation of $(1,2,3)$. (For a proof,
contemplate Fig.~\ref{fig:ideal_triang}.)

Now consider a decorated ideal quadrilateral as shown in
Fig.~\ref{fig:ptolemy}. 
\begin{figure}
  \begin{minipage}{0.45\linewidth}
    \labellist
    \small\hair 2pt \pinlabel {$a$} [t] at 66 46 \pinlabel {$b$} [l]
    at 107 59 \pinlabel {$c$} [bl] at 94 78 \pinlabel {$d$} [br] at 45
    66 \pinlabel {$e$} [br] at 71 60 \pinlabel {$f$} [tr] <0pt,4pt> at
    89 55
    \endlabellist
    \centering
    \includegraphics{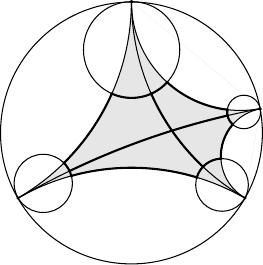}
    \caption{Ptolemy relation}
    \label{fig:ptolemy}
  \end{minipage}%
  \begin{minipage}{0.55\linewidth}
    \labellist
    \small\hair 2pt
    \pinlabel {$a$} [l] at 63 55
    \pinlabel {$a'$} [b] at 76 67
    \pinlabel {$b$} [tl] at 82 54
    \pinlabel {$b$} [br] at 43 72
    \pinlabel {$c$} [bl] at 84 88
    \pinlabel {$c$} [tr] at 42 41
    \endlabellist
    \centering
    \includegraphics{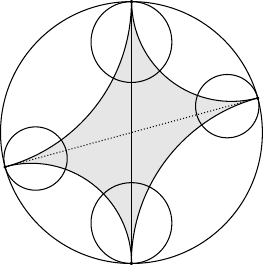}
    \caption{Triangulations $T$ and~$T'$ of a punctured torus}
    \label{fig:torus}    
  \end{minipage}
\end{figure}
It can be decomposed into two decorated ideal
triangles in two ways. The six weights $a$, $b$, $c$, $d$, $e$, $f$
are related by the Ptolemy relation
\begin{equation}
  \label{eq:ptolemy}
  ef=ac+bd.
\end{equation}
It is straightforward to derive this equation using the
relations~\eqref{eq:c}.

\section{Triangulations of the modular torus and Markov's equation}
\label{sec:triangulations}

In this section, we review Penner's~\cite{penner87,penner12}
geometric interpretation of Markov's equation~\eqref{eq:markov}, which
is summarized in Prop.~\ref{prop:triangs}. The involutions
$\sigma_{k}$ were defined in Sec.~\ref{sec:mostirrational}, see
equation~\eqref{eq:sigma_k}. The \emph{modular torus} is the orbit
space 
\begin{equation*}
  M=H^{2}/G, 
\end{equation*}
where~$G$ is the group of orientation preserving
hyperbolic isometries generated by
\begin{equation}
  \label{eq:AB}
  A(z) = \frac{z-1}{-z+2},\qquad B(z) = \frac{z+1}{z+2}. 
\end{equation}
Figure~\ref{fig:modular_t} shows a fundamental domain.
\begin{figure}
  \labellist
  \small\hair 2pt
  \pinlabel {$-1$} [t] at 36 8
  \pinlabel {$0$} [t] at 94 8
  \pinlabel {$1$} [t] at 152 8
  \pinlabel {$B$} [bl] at 65 97
  \pinlabel {$A$} [br] at 122 97
  \endlabellist
  \hfill
  \includegraphics{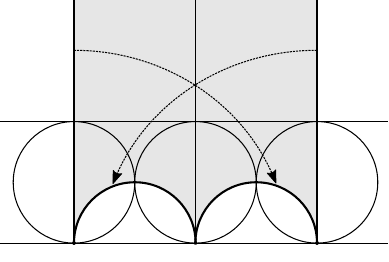}%
  \hspace*{\fill}
  \caption{The modular torus}
  \label{fig:modular_t}
\end{figure}
The group $G$ is the commutator subgroup of the modular group
$\PSLTZ$, and the only subgroup of $\PSLTZ$ that has a once punctured
torus as orbit space. It is a normal subgroup of $\PSLTZ$ with index
six, and the quotient group $\PSLTZ/G$ is the group of orientation
preserving isometries of the modular torus $M$. It is also symmetric
with respect to six reflections, so the isometry group has
in total twelve elements.

\begin{proposition}[Markov triples and ideal triangulations]
  \label{prop:triangs}
  (i) A triple $\tau=(a,b,c)$ of positive integers is a Markov triple
  if and only if there is an ideal triangulation of the decorated
  modular torus whose three edges have the weights $a$, $b$, and
  $c$. This triangulation is unique up to the $12$-fold symmetry of the
  modular torus.

  (ii) If $T$ is an ideal triangulation of the decorated modular torus
  with edge weights $\tau=(a,b,c)$, and if $T'$ is an ideal
  triangulation obtained from $T$ by performing a single edge flip,
  then the edge weights of $T'$ are $\tau'=\sigma_{k}\tau$, with
  $k\in\{1,2,3\}$ depending on which edge was flipped.
\end{proposition}

To understand the logical connections, it makes sense to consider not
only the modular torus but arbitrary once punctured hyperbolic tori.

A \emph{once punctured hyperbolic torus} is a torus with one point
removed, equipped with a complete metric of constant curvature $-1$
and finite volume. For example, one obtains a once punctured
hyperbolic torus by gluing two congruent decorated ideal triangles
along their edges in such a way that the horocycles fit together.
Conversely, every ideal triangulation of a hyperbolic torus with one
puncture decomposes it into two ideal triangles.

A \emph{decorated once punctured hyperbolic torus} is a once punctured
hyperbolic torus together with a choice of horocycle at the
cusp. Thus, a triple of weights $(a,b,c)\in\R_{>0}^{3}$ determines a
decorated once punctured hyperbolic torus up to isometry, together
with an ideal triangulation. Conversely, a decorated once punctured
hyperbolic torus together with an ideal triangulation determines such
a triple of edge weights.

Consider a decorated once punctured hyperbolic torus with an ideal
triangulation $T$ with edge weights $(a,b,c)\in\R_{>0}^{3}$. By
equation~\eqref{eq:c}, the total length of the horocycle is
\begin{equation*}
  \ell = 2\big(\frac{a}{bc}+\frac{b}{ca}+\frac{c}{ab}\big).
\end{equation*}
This equation is equivalent to 
\begin{equation*}
  a^{2}+b^{2}+c^{2}=\frac{\ell}{2}\,abc.
\end{equation*}
Thus, the weights satisfy Markov's equation~\eqref{eq:markov} (not
considered as a Diophantine equation) if and only if the horocycle has
length $\ell=6$. From now on, we assume that this is the case:
We decorate all once punctured hyperbolic tori with the horocycle
of length $6$.

Let $T'$ be the ideal triangulation obtained from $T$ by flipping the
edge with weight $a$, i.e., by replacing this edge with the other
diagonal in the ideal quadrilateral formed by the other edges (see
Fig.~\ref{fig:torus}). By equation~\eqref{eq:a_prime} and Ptolemy's
relation~\eqref{eq:ptolemy}, the edge weights of $T'$ are
$(a',b,c)=\sigma_{1}(a,b,c)$. Of course, one obtains analogous
equations if a different edge is flipped.

The modular torus $M$, decorated with a horocycle of length $6$, is
obtained by gluing two decorated ideal triangles with weights
$(1,1,1)$. Lifting this triangulation and decoration to the hyperbolic
plane, one obtains the Farey tessellation with Ford circles
(Fig.~\ref{fig:farey}).  This implies that for every Markov triple
$(a,b,c)$ there is an ideal triangulation of the decorated modular
torus with edge weights $a$, $b$, $c$. To see this, follow the path in
the Markov tree leading from $(1,1,1)$ to $(a,b,c)$ and perform the
corresponding edge flips on the projected Farey tessellation.

On the other hand, the flip graph of a complete hyperbolic surface
with punctures is also connected \cite{hatcher91}
\cite[p.~36ff]{mosher88}. The \emph{flip graph} has the ideal
triangulations as vertices, and edges connect triangulations related
by a single edge flip. (Since we are only interested in a once
punctured torus, invoking this general theorem is somewhat of an
overkill.) This implies the converse statement: If $a$, $b$, $c$ are
the weights of an ideal triangulation of the modular torus, then
$(a,b,c)$ is a Markov triple.

Note that there is only one ideal triangulation of the modular torus
with weights $(1,1,1)$, i.e., the triangulation that lifts to the
Farey tessellation. The symmetries of the modular torus permute its
edges. Since the Markov tree and the flip graph are isomorphic, this
implies that two triangulations with the same weights are related by
an isometry of the modular torus. Altogether, one obtains
Proposition~\ref{prop:triangs}.

\section{Geodesics crossing a decorated ideal triangle}
\label{sec:crossing}

For the proof of Markov's theorem in Sec.~\ref{sec:markovproof}, we
need to know how far a geodesic crossing a decorated ideal triangle
can stay away from the horocycles at the vertices. To prove Hurwitz's
theorem (see Sec.~\ref{sec:hurwitz}), it was enough to consider a
triangle decorated with pairwise tangent horocycles. In this section,
we consider the general case, more precisely, the following geometric
optimization problem:

\begin{problem}
  \label{prob:crossing}
  \emph{Given} a decorated ideal triangle with two sides, say $a_{1}$
  and $a_{2}$, designated as ``legs'', and the third side, say
  $a_{3}$, designated as ``base''. \emph{Find}, among all geodesics
  intersecting both legs, a geodesic that maximizes the minimum of
  signed distances to the three horocycles at the vertices.
\end{problem}

It makes sense to consider the corresponding optimization problem for
euclidean triangles: Which straight line crossing two given legs has
the largest distance to the vertices? The answer depends on whether or
not an angle at the base is obtuse. For decorated
ideal triangles, the situation is completely analogous. We say that a
geodesic \emph{bisects} a side of a decorated ideal triangle if it
intersects the side in the point at equal distance to the two
horocycles at the ends of the side.

\begin{sidewaysfigure}
  \hfill
  \labellist
  \small\hair 2pt
  \pinlabel {$v_{2}$} [t] at 198 9
  \pinlabel {$v_{1}$} [t] at 67 9
  \pinlabel {$v_{3}=\infty$} [l] at 130 200
  \pinlabel {$a_{1}$} [l] at 197 120
  \pinlabel {$a_{2}$} [r] at 66 130
  \pinlabel {$a_{3}$} [t] at 133 74
  \pinlabel {$h_{2}$} [br] at 220 21
  \pinlabel {$h_{1}$} [bl] at 32 29
  \pinlabel {$h_{3}$} [b] at 32 156
  \pinlabel {$g$} [r] at 11 50
  \pinlabel {$P_{2}$} [bl] at 228 51
  \pinlabel {$P_{1}$} [br] at 30 83
  \pinlabel {$P_{3}$} [tl] <-2pt,0pt> at 121 126
  \pinlabel {$c_{2}$} [br] at 186 63
  \pinlabel {$c_{1}$} [bl] at 84 81
  \pinlabel {$c_{3}$} [b] at 129 159
  \pinlabel {$s_{2}$} [t] at 93 152
  \pinlabel {$s_{2}$} [br] at 54 82
  \pinlabel {$s_{1}$} [t] at 156 152
  \pinlabel {$s_{1}$} [bl] <-1pt, -2pt> at 212 61
  \pinlabel {$s_{3}$} [tr] at 65 79
  \pinlabel {$s_{3}$} [tl] at 199 58
  \pinlabel {$x_{1}$} [t] <1pt,0pt> at 2 9
  \pinlabel {$x_{0}$} [t] <1pt,0pt> at 120 9
  \pinlabel {$x_{2}$} [t] <1pt,0pt> at 236.5 9
  \pinlabel {$r$} [br] at 145 38
  \pinlabel {$\,1$} [l] at 261 156
  \pinlabel {$\frac{1}{a_{1}}$} [l] at 261 97
  \pinlabel {$\frac{1}{a_{1}^{2}}$} [l] at 261 61
  \endlabellist
  \includegraphics{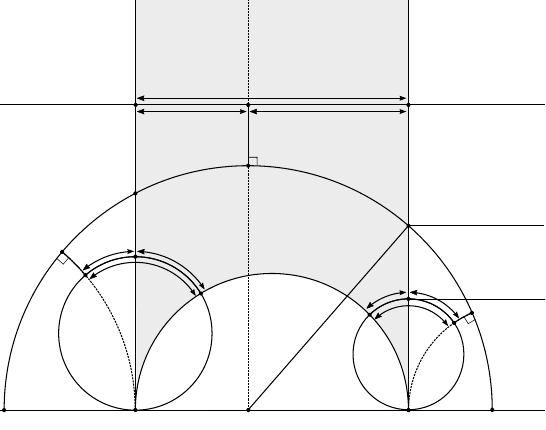}
 \hfill
\labellist
\small\hair 2pt
 \pinlabel {$v_{2}$} [t] at 158 9
 \pinlabel {$v_{1}$} [t] at 68 9
 \pinlabel {$v_{3}=\infty$} [t] at 111 204
 \pinlabel {$h_{2}$} [r] at 148 18
 \pinlabel {$h_{1}$} [l] at 27 33
 \pinlabel {$h_{3}$} [b] at 17 156
 \pinlabel {$g$} [b] at 13 119
 \pinlabel {$P_{2}$} [l] at 169 23
 \pinlabel {$P_{1}$} [b] at 82 123
 \pinlabel {$P_{3}$} [t] at 53 125
 \pinlabel {$a_{1}$} [l] at 157 96
 \pinlabel {$a_{2}$} [br] at 67 129
 \pinlabel {$a_{3}$} [t] at 124 52
 \pinlabel {$c_{3}$} [b] at 106 159
 \pinlabel {$s_{2}$} [b] at 60 159
 \pinlabel {$s_{1}$} [t] at 98 151
 \pinlabel {$c_{1}$} [bl] at 102 88
 \pinlabel {$s_{3}$} [tr] at 99 84
 \pinlabel {\footnotesize $s_{2}$} [t] at 70.5 97
 \pinlabel {\tiny $c_{2}$} [b] at 153 29
 \pinlabel {\tiny $s_{1}$} [b] at 163 27
 \pinlabel {\tiny $s_{3}$} [t] at 159.5 26
\endlabellist
 \includegraphics{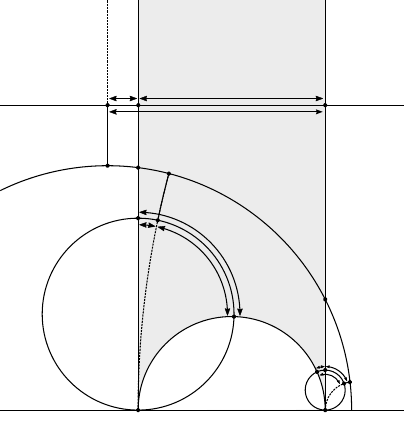}
 \hspace*{\fill}
 \caption{Decorated ideal triangle (shaded) and geodesic $g$ through
   the midpoints of sides $a_{1}$ and $a_{2}$. \emph{Left:}
   Inequalities~\eqref{eq:acute} are strictly satisfied and $P_{3}$
   lies strictly between $P_{1}$ and $P_{2}$. (The height marks on the
   right margin belong to the proof of Proposition~\ref{prop:crossing2}.)
   \emph{Right:} $a_{1}^{2}>a_{2}^{2}+a_{3}^{2}$ and $P_{1}$ lies
   strictly between $P_{3}$ and $P_{2}$.}
 \label{fig:crossing}
\end{sidewaysfigure}

\begin{proposition}
  \label{prop:crossing1}
  Consider a decorated ideal triangle with horocycles $h_{1}$,
  $h_{2}$, $h_{3}$, and let $a_{1}$, $a_{2}$, $a_{3}$ denote both the
  sides and their weights (see Fig.~\ref{fig:crossing} for notation).

  (i) If
  \begin{equation}
    \label{eq:acute}
    a_{1}^{2}\leq a_{2}^{2}+a_{3}^{2}
    \quad\text{and}\quad
    a_{2}^{2}\leq a_{1}^{2}+a_{3}^{2},
  \end{equation}
  then the geodesic $g$ bisecting the sides $a_{1}$ and $a_{2}$ is the
  unique solution of Problem~\ref{prob:crossing}. 

  (ii) If, for $(j,k)\in\{(1,2),(2,1)\}$,
  \begin{equation}
    \label{eq:obtuse}
    a_{j}^{2}\geq a_{k}^{2}+a_{3}^{2},
  \end{equation}
  then the perpendicular bisector $g'$ of side $a_{k}$ is the unique
  solution of Problem~\ref{prob:crossing}. In this case, the minimal
  distance is attained for $h_{j}$ and $h_{3}$,
  \begin{equation}
    \label{eq:perp_bisector}
    d(h_{j},g')=d(h_{3},g')=\frac{\alpha_{k}}{2}\leq d(h_{k},g').
  \end{equation}
\end{proposition}

In the proof of Markov's theorem (Sec.~\ref{sec:markovproof}), the
base $a_{3}$ will always be a largest side, so only part (i) of
Proposition~\ref{prop:crossing1} is needed. We will also need
some equations for the geodesic bisecting two sides, which we collect
in Proposition~\ref{prop:crossing2}. 

\begin{proof}[Proof of Proposition~\ref{prop:crossing1}]
  1. The geodesic $g$ has equal distance from all three
  horocycles. Indeed, because of the $180^{\circ}$ rotational symmetry
  around the intersection point, any geodesic bisecting a side has
  equal distance from the two horocycles at the ends.

  2. For $k\in\{1,2,3\}$ let $P_{k}$ be the foot of the perpendicular
  from vertex $v_{k}$ to the geodesic $g$ bisecting $a_{1}$ and
  $a_{2}$ (see Fig.~\ref{fig:crossing}). If $P_{3}$ lies strictly
  between $P_{1}$ and $P_{2}$ (as in Fig.~\ref{fig:crossing}, left),
  then $g$ is the unique solution of Problem~\ref{prob:crossing}. Any
  other geodesic crossing $a_{1}$ and $a_{2}$ also crosses at least
  one of the rays from $P_{k}$ to $v_{k}$, and is therefore closer to
  at least one of the horocycles.

  3. If $P_{1}$ lies strictly between $P_{3}$ and $P_{2}$ (as in
  Fig.~\ref{fig:crossing}, right) then the unique solution of
  Problem~\ref{prob:crossing} is the perpendicular bisector of
  $a_{2}$. Its signed distance to the horocycles $h_{1}$ and $h_{3}$
  is half the truncated length of side $a_{2}$. Any other geodesic
  crossing $a_{2}$ is closer to at least one of its horocycles. The
  signed distance of $g$ and the horocycle $h_{1}$ is larger. The case
  when $P_{1}$ lies strictly between $P_{3}$ and $P_{2}$ is treated in
  the same way.

  5. If $P_{2}=P_{3}$ (or $P_{1}=P_{3}$) then the geodesic $g$ with
  equal distance to all horocycles is simultaneously the perpendicular
  bisector of side $a_{2}$ (or $a_{1}$). 

  6. It remains to show that the order of the points $P_{k}$ on $g$
  depends on whether the weights satisfy the
  inequalities~\eqref{eq:acute} or one of the
  inequalities~\eqref{eq:obtuse}. To this end, let $s_{1}$ be the
  distance from the side $a_{1}$ to the ray $P_{3}v_{3}$, measured
  along the horocycle~$h_{3}$ in the direction from $a_{1}$ to
  $a_{2}$. Similarly, let $s_{2}$ be the distance from the side
  $a_{2}$ to the ray $P_{3}v_{3}$, measured along the horocycle
  $h_{3}$ in the direction from $a_{2}$ to $a_{1}$. So $s_{1}$ and
  $s_{2}$ are both positive if and only if $P_{3}$ lies strictly
  between $P_{1}$ and $P_{2}$. But if, for example, $P_{1}$ lies
  between $P_{3}$ and $P_{2}$ as in Fig.~\ref{fig:crossing}, right,
  then $s_{2}<0$. By symmetry, $s_{1}$ is also the distance from
  $a_{1}$ to $P_{2}v_{2}$, measured along $h_{2}$ in the direction
  away from $a_{3}$. Similarly, $s_{2}$ is also the distance between
  $a_{2}$ and $P_{1}v_{1}$ along $h_{1}$. Finally, let $s_{3}>0$ be
  the equal distances between $a_{3}$ and $P_{1}v_{1}$ along $h_{1}$,
  and between $a_{3}$ and $P_{2}v_{2}$ along $h_{2}$. Now
  \begin{equation*}
    c_{1}=-s_{2}+s_{3},\quad c_{2}=-s_{1}+s_{3},\quad c_{3}=s_{1}+s_{2}
  \end{equation*}
  implies
  \begin{equation}
    \label{eq:2s1}
    2s_{1}=c_{1}-c_{2}+c_{3}
    \overset{\eqref{eq:c}}{=}
    \frac{a_{1}}{a_{2}a_{3}}-\frac{a_{2}}{a_{3}a_{1}}+\frac{a_{3}}{a_{1}a_{2}}
    =\frac{a_{1}^{2}-a_{2}^{2}+a_{3}^{2}}{a_{1}a_{2}a_{3}}
  \end{equation}
  and similarly
  \begin{equation*}
    2s_{2}=\frac{-a_{1}^{2}+a_{2}^{2}+a_{3}^{2}}{a_{1}a_{2}a_{3}}.
  \end{equation*}
  Hence, $P_{3}$ lies in the closed interval between $P_{1}$ and
  $P_{2}$ if and only if inequalities~\eqref{eq:acute} are satisfied.
  The other cases are treated similarly.
\end{proof}

\begin{remark}
  The above proof of Proposition~\ref{prop:crossing1} is nicely
  intuitive. A more analytic proof may be obtained as follows. First,
  show that for all geodesics intersecting $a_{1}$ and $a_{2}$, the
  signed distances $u_{1}$, $u_{2}$, $u_{3}$ to the horocycles satisfy
  the equation
  \begin{equation}
    \label{eq:urelation}
    (c_{1}u_{1}+c_{2}u_{2}+c_{3}u_{3})^{2}-4c_{1}c_{2}u_{1}u_{2}-4 =0
  \end{equation}
  It makes sense to consider the special case $a_{1}=a_{2}=a_{3}=1$
  first, because the general equation~\eqref{eq:urelation} can easily
  be derived from the simpler one. Then consider the necessary conditions
  for a local maximum of $\min(u_{1},u_{2},u_{3})$ under the
  constraint~\eqref{eq:urelation}: If a maximum is attained with
  $u_{1}=u_{2}=u_{3}$, then the three partial derivatives of the left
  hand side of~\eqref{eq:urelation} are all $\geq 0$ or all $\leq 0$.
  If a maximum is attained with $u_{1}=u_{2}<u_{3}$, then this sign
  condition holds for the first two derivatives, and similarly for the
  other cases.
\end{remark}

\begin{proposition}
  \label{prop:crossing2}
  Let $g$ be the geodesic bisecting sides $a_{1}$ and $a_{2}$ of a
  decorated ideal triangle as shown in
  Fig.~\ref{fig:crossing}. (Inequalities~\eqref{eq:acute} may hold or
  not.) Then the common signed distance of $g$ and the horocycles is
  \begin{equation*}
    d(h_{1},g)=d(h_{2},g)=d(h_{3},g)= -\log r,
  \end{equation*}
  where
  \begin{equation}
    \label{eq:r}
    r = \sqrt{\frac{\delta^{2}}{4}-\frac{1}{a_{3}^{2}}},
  \end{equation}
  and $\delta$ is the sum of the lengths of the horocyclic arcs,
  \begin{equation}
    \label{eq:delta}
    \delta = c_{1}+c_{2}+c_{3} = \frac{a_{1}}{a_{2}a_{3}}
    +\frac{a_{2}}{a_{3}a_{1}}+\frac{a_{3}}{a_{1}a_{2}}.
  \end{equation}
  Moreover, suppose the vertices are 
  \begin{equation}
    \label{eq:v_assumption}
    v_{1}<v_{2},\quad v_{3}=\infty, 
  \end{equation}
  and the horocycle $h_{3}$ has height $1$. Then the ends $x_{1,2}$
  of $g$ are
  \begin{equation}
    \label{eq:x1x2}
    x_{1,2}=x_{0}\pm r,
  \end{equation}
  where
  \begin{equation}
    \label{eq:x0alt}
    x_{0}=v_{2}+\frac{a_{2}}{a_{3}a_{1}}-\frac{\delta}{2}
  \end{equation}
\end{proposition}

\begin{proof}
  Assuming~\eqref{eq:v_assumption} and $h_{3}=h(1,0)$, let
  $x_{0}=v_{2}-s_{1}$.  Then the proposition follows
  from~\eqref{eq:2s1}, some easy hyperbolic geometry, Pythagoras'
  theorem, and simple algebra (see Fig.~\ref{fig:crossing}).
\end{proof}

\section{Simple closed geodesics and ideal arcs}
\label{sec:arcs_geodesics}

In this section, we collect some topological facts about simple closed
geodesics and ideal arcs that we will use in the proof of Markov's
theorem (Sec.~\ref{sec:markovproof}). They are probably well known,
but we indicate proofs for the reader's convenience.

An \emph{ideal arc} in a complete hyperbolic surface with cusps is a
simple geodesic connecting two punctures or a puncture with
itself. The edges of an ideal triangulation are ideal arcs, and every
ideal arc occurs in an ideal triangulation. (In fact, ideal
triangulations are exactly the maximal sets of non-intersecting ideal
arcs.) Here, we are only interested in a once punctured hyperbolic
torus. In this case, every ideal triangulation containing a fixed
ideal arc can be obtained from any other such triangulation by
repeatedly flipping the remaining two edges. Ideal arcs play an
important role in the following section because they are in one-to-one
correspondence with the simple closed geodesics
(Proposition~\ref{prop:arcs_geodesics}), and the simple closed
geodesics are the geodesics that stay farthest away from the puncture
(Proposition~\ref{prop:markov_forms_geo}).

\begin{proposition}
  \label{prop:arcs_geodesics}
  Consider a fixed once punctured hyperbolic torus.
  
  (i) For every ideal arc~$c$, there is a unique simple closed
  geodesic~$g$ that does not intersect $c$. 

  (ii) Every other geodesic not intersecting $c$ has either two ends
  in the puncture, or one end in the puncture and the other end
  approaching the closed geodesic $g$.

  (iii) If $a$, $b$, $c$ are the edges of an ideal triangulation $T$,
  then the simple closed geodesic $g$ that does not intersect $c$
  intersects each of the two triangles of $T$ in a geodesic segment
  bisecting the edges $a$ and $b$.

  (iv) For every simple closed geodesic~$g$, there is a unique ideal
  arc~$c$ that does not intersect~$g$.
\end{proposition}

\begin{remark}
  \label{rem:involution}
  Speaking of edge midpoints implies an (arbitrary) choice of a
  horocycle at the cusp. In fact, the edge midpoints of a triangulated
  once punctured torus are distinguished without any choice of
  triangulation. They are the three fixed points of an orientation
  preserving isometric involution. Every ideal arc passes through one
  of these points.
\end{remark}

\begin{proof}
  (i) Cut the torus along the ideal arc $c$. The result is a
  hyperbolic cylinder as shown in Fig.~\ref{fig:cylinder} (left). 
  \begin{figure}
    \centering
    \hfill
    \labellist
    \small\hair 2pt
    \pinlabel {$c$} [t] at 29 73
    \pinlabel {$g$} [t] at 29 44
    \pinlabel {$c$} [b] at 29 13
    \endlabellist
    \includegraphics[scale=1.2]{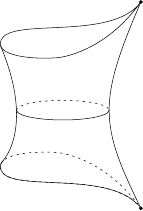}
    \hfill
    \labellist
    \small\hair 2pt
    \pinlabel {$g$} [t] at 30 69
    \pinlabel {$c$} [t] at 30 40
    \pinlabel {$g$} [b] at 30 0
    \endlabellist
    \includegraphics[scale=1.2]{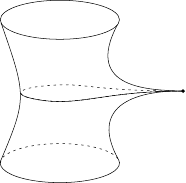}
    \hspace*{\fill}
    \caption{Cutting a punctured torus along an ideal arc (left) and
      along a simple closed geodesic (right).}
    \label{fig:cylinder}
  \end{figure}
  Both boundary curves are complete geodesics with both ends in the
  cusp, which is now split in two. There is up to orientation a unique
  non-trivial free homotopy class that contains simple curves, and
  this class contains a unique simple closed geodesic.

  (ii) Consider the universal cover of the cylinder in the hyperbolic
  plane.

  (iii) An ideal triangulation of a once punctured torus is symmetric
  with respect to a $180^{\circ}$ rotation around the edge
  midpoints. (This is the involution mentioned in
  Remark~\ref{rem:involution}.) It swaps the geodesic segments
  bisecting edges $a$ and $b$ in the two ideal triangles, so they
  connect smoothly. Hence they form a simple closed geodesic, which
  does not intersect $c$.

  (iv) Cut the torus along the simple closed geodesic $g$. The result
  is a cylinder with a cusp and two geodesic boundary circles, as
  shown in Fig.~\ref{fig:cylinder} (right). Fill the puncture and take
  it as base point for the homotopy group. There is up to orientation
  a unique non-trivial homotopy class containing simple closed curves
  and this class contains a unique ideal arc.
\end{proof}

\section{Proof of Markov's theorem}
\label{sec:markovproof}

In this section, we put the pieces together to prove both versions of
Markov's theorem. The quadratic forms version follows from
Proposition~\ref{prop:markov_forms_geo}. The Diophantine approximation
version follows from Proposition~\ref{prop:markov_forms_geo} together
with Proposition~\ref{prop:markov_dioph_geo}.

Two geodesics in the hyperbolic plane are \emph{$\GLTZ$-related} if,
for some $A\in\GLTZ$, the hyperbolic isometry $M_{A}$ maps one to the
other.

\begin{proposition}
  \label{prop:markov_forms_geo}
  Let $g$ be a complete geodesic in the hyperbolic plane, and
  let $\pi(g)$ be its projection to the modular torus. Then the
  following three statements are equivalent:
  
  \begin{compactenum}[(a)]
  \item $\pi(g)$ is a simple closed geodesic.
  \item There is a Markov triple $(a,b,c)$ so that for one (hence
    any) choice of integers $p_{1}$, $p_{2}$
    satisfying~\eqref{eq:p1p2}, the geodesic $g$ is $\GLTZ$-related
    to the geodesic ending in $x_{0}\pm r$ with $x_{0}$ and $r$
    defined by~\eqref{eq:markov_x0} and~\eqref{eq:markov_r}. 
  \item The greatest lower bound for the signed distances of $g$ and a
    Ford circle is greater than $-\log\frac{3}{2}$.
  \end{compactenum}
  If $g$ satisfies one (hence all) of the statements (a), (b), (c),
  then
  \begin{compactenum}[(a)]
    \setcounter{enumi}{3}
  \item the minimal signed distance of $g$ and a Ford circle is $-\log r$,
  \item among all Markov triples $(a,b,c)$ that verify (b), there is a
    unique sorted Markov triple.
  \end{compactenum}
\end{proposition}

\begin{proof}
  ``$\textit{(a)}\Rightarrow\textit{(b)}$'': If $\pi(g)$ is a
  simple closed geodesic, then there is a unique ideal arc $c$ not
  intersecting $\pi(g)$ (Proposition~\ref{prop:arcs_geodesics}
  (iv)). Pick an ideal triangulation $T$ of the modular torus that
  contains $c$, and let $a$ and $b$ be the other edges. By
  Proposition~\ref{prop:triangs}, $(a,b,c)$ is a Markov triple. (We
  use the same letters to denote both ideal arcs and their weights.)
  The geodesic $\pi(g)$ intersects each of the two triangles
  of $T$ in a geodesic segment bisecting the edges~$a$ and~$b$
  (Proposition~\ref{prop:arcs_geodesics}~(iii)).

  Now let $p_{1}$, $p_{2}$ be integers satisfying~\eqref{eq:p1p2} and
  consider the decorated ideal triangle in~$H^{2}$ with vertices
  \begin{equation}
    \label{eq:v1v2}
    v_{1}=\frac{p_{1}}{b},\quad
    v_{2}=\frac{p_{2}}{a},\quad
    v_{3}=\infty,
  \end{equation}
  and their respective Ford
  circles
  \begin{equation}
    h_{1}=h(p_{1},b),\quad h_{2}=h(p_{2},a),\quad h_{3}=h(1,0).
  \end{equation}
  Such integers $p_{1}$, $p_{2}$ exist because the numbers $a$, $b$,
  $c$ of a Markov triple are pairwise coprime. Moreover, this implies
  that the fractions in~\eqref{eq:v1v2} are reduced, and $v_{1}$ and
  $v_{2}$ are determined up to addition of a common integer. By
  Proposition~\ref{prop:dhh}, this decorated ideal triangle has edge
  weights
  \begin{equation}
    \label{eq:a1a2a3}
    a_{1}=a,\quad
    a_{2}=b,\quad 
    a_{3}=c
  \end{equation}
  (see Fig.~\ref{fig:crossing} for notation). 

  Conversely, every ideal triangle
  $\tilde{v}_{1}\tilde{v}_{2}\tilde{v}_{3}$ with
  $\tilde{v}_{3}=\infty$ and rational $\tilde{v}_{1}$,
  $\tilde{v}_{2}$, that is decorated with the respective Ford circles,
  has weights~\eqref{eq:a1a2a3}, and satisfies
  $\tilde{v}_{1}<\tilde{v}_{2}$ is obtained this way. (To get the
  triangles with $\tilde{v}_{1}>\tilde{v}_{2}$, change $c$ to $-c$ in
  equation~\eqref{eq:p1p2}.)  This implies that any lift of a triangle
  of $T$ to the hyperbolic plane is $\GLTZ$-related to
  $v_{1}v_{2}v_{3}$. Use Proposition~\ref{prop:crossing2} with
  $\delta=3$ to deduce that $g$ is $\GLTZ$-related to the geodesic
  ending in $x_{0}\pm r$.

  ``$\textit{(b)}\Rightarrow\textit{(d)}$'': Let $\hat{T}$ be the
  lift of the triangulation $T$ to $H^{2}$. The geodesic~$g$ crosses
  an infinite strip of triangles of $\hat{T}$. By
  Proposition~\ref{prop:crossing2}, the signed distance of $g$ and any
  Ford circle centered at a vertex incident with this strip is
  $-\log r$. We claim that the signed distance to any other Ford
  circle is larger. To see this, consider a vertex
  $v\in\Q\cup\{\infty\}$ that is not incident with the triangle strip,
  and let $\rho$ be a geodesic ray from $v$ to a point $p\in g$. Note
  that the projected ray $\pi(\rho)$ intersects $\pi(g)$ at least once
  before it ends in $\pi(p)$, and that the signed distance to the
  first intersection is at least $-\log r$.

  ``$\textit{(b)}\land\textit{(d)}\Rightarrow\textit{(c)}$'':
  This follows directly from
  $r=\sqrt{\frac{9}{4}-\frac{1}{c^{2}}}<\frac{3}{2}$.

  ``$\textit{(c)}\Rightarrow\textit{(a)}$'': We will show the
  contrapositive: If the geodesic $g$ does not project to a simple
  closed geodesic, then there is a Ford circle with signed distance
  smaller than $-\log\frac{3}{2}+\epsilon$, for every $\epsilon>0$.

  There is nothing to show if at least one end of $g$ is in
  $\Q\cup\{\infty\}$ because then the Ford circle at this end has
  signed distance $-\infty$. So assume $g$ does not project to a
  simple closed geodesic and both ends of~$g$ are irrational.

  We will recursively define a sequence $(T_{n})_{n\geq 0}$
  of ideal triangulations of the modular torus, with edges labeled
  $a_{n}$, $b_{n}$, $c_{n}$, such that the following holds:

  \begin{compactenum}[(1)]
  \item The geodesic $\pi(g)$ has at least one pair of consecutive
    intersections with the edges $a_{n}$, $b_{n}$.
  \item The edge weights, which we also denote by $a_{n}$, $b_{n}$,
    $c_{n}$, satisfy 
    \begin{equation*}
      a_{n}\leq b_{n}\leq c_{n},
    \end{equation*}
    so that $(a_{n},b_{n},c_{n})$ is a sorted Markov triple. 
  \item
    $
      c_{n+1}>c_{n}
    $
  \end{compactenum}

  This proves the claim, because Propositions~\ref{prop:crossing1}
  and~\ref{prop:crossing2} imply that for each~$n$, there is a
  horocycle with signed distance to $g$ less than 
  \smash{$
    -\frac{1}{2}\log\big(\frac{9}{4}-\frac{1}{c_{n}^{2}}\big),
  $}
  which tends to \smash{$-\log\frac{3}{2}$} from above as
  $n\rightarrow\infty$.

  To define the sequence $(T_{n})$, let $T_{0}$ be the triangulation
  with edge weights $(1,1,1)$, with edges labeled so that (1) holds.

  Suppose the triangulation $T_{n}$ with labeled edges is already
  defined for some $n\geq 0$. Define the labeled triangulation
  $T_{n+1}$ as follows. Since $\pi(g)$ is not a simple closed
  geodesic, it intersects all three edges. Because $g$ has
  an irrational end (in fact, both ends are assumed to be irrational),
  there are infinitely many edge intersections. Hence, there is pair of
  intersections with $a_{n}$ and $b_{n}$ next to an intersection with
  $c_{n}$. If the sequence of intersections is $a_{n}b_{n}c_{n}$, let
  $T_{n+1}$ be the triangulation with edges
  \begin{equation*}
    (a_{n+1},b_{n+1},c_{n+1})=(a_{n},c_{n},b_{n}'),
  \end{equation*}
  and if the sequence is $b_{n}a_{n}c_{n}$, let $T_{n+1}$ be the
  triangulation with 
  \begin{equation*}
    (a_{n+1},b_{n+1},c_{n+1})=(b_{n},c_{n},a'_{n}),
  \end{equation*}
  where $a_{n}'$ and $b_{n}'$ are the ideal arcs obtained by flipping
  the edges $a_{n}$ or $b_{n}$ in $T_{n}$, respectively. By induction
  on $n$, one sees that (1), (2), (3) are satisfied for all $n\geq 0$.

  ``$\textit{(a)}\land\textit{(b)}\Rightarrow\textit{(e)}$'': The
  Markov triples $(a,b,c)$ verifying (b) are precisely the triples of
  edge weights of ideal triangulations containing the ideal arc $c$
  not intersecting $\pi(g)$. The triangulations containing the ideal
  arc $c$ form a doubly infinite sequence in which neighbors are
  related by a single edge flip fixing $c$. In this sequence, there is
  a unique triangulation for which the weight $c$ is largest.
\end{proof}

\begin{proposition}
  \label{prop:markov_dioph_geo}
  Let $g$ be a complete geodesic in the hyperbolic plane, and let
  $X\subset\R\setminus\Q$ be the set of ends of lifts of simple closed
  geodesics in the modular torus. Then the following two statements
  are equivalent:
  
  \begin{compactenum}[(i)]
  \item The ends of $g$ are contained in $\Q\cup\{\infty\}\cup X$.
  \item For some $M>-\log\frac{3}{2}$ there are only finitely many
    (possibly zero) Ford circles~$h$ with signed distance $d(g,h)<M$.
  \end{compactenum}
\end{proposition}

\begin{proof}
  ``$\textit{(i)}\Rightarrow\textit{(ii)}$'': Consider the ends
  $x_{k}$ of $g$, $k\in\{1,2\}$. 

  If $x_{k}\in\Q\cup\{\infty\}$, then $g$ contains a ray $\rho_{k}$
  that is contained inside the Ford circle at $x_{k}$. In this case,
  let $M_{k}=0$.

  If $x_{k}\in X$, then $x_{k}$ is also the end of a geodesic
  $\tilde{g}$ that projects to a simple closed geodesic in the modular
  torus. By Proposition~\ref{prop:markov_forms_geo},
  $\inf d(h,\tilde{g})>-\log\frac{3}{2}$, where the infimum is taken
  over all Ford circles $h$. Since $g$ and $\tilde{g}$ converge
  at~$x_{k}$, there is a constant $M_{k}>-\log\frac{3}{2}$ and a ray
  $\rho_{k}$ contained in $g$ and ending in $x_{k}$ such that
  $d(h,\rho_{k})>M_{k}$ for all Ford circles $h$.

  The part of $g$ not contained in $\rho_{1}$ or $\rho_{2}$ is empty
  or of finite length, so it can intersect the interiors of at most
  finitely many Ford circles. This implies (ii) with
  $M=\min(M_{1},M_{2})$.

  ``$\textit{(ii)}\Rightarrow\textit{(i)}$'': To show the
  contrapositive, assume (i) is false: At least one end of $g$ is
  irrational but not the end of a lift of a simple closed geodesic in
  the modular torus. This implies that the projection $\pi(g)$
  intersects every ideal arc in the modular torus infinitely many
  times. Adapt the argument for the implication
  ``$\text{(c)}\Rightarrow\text{(a)}$'' in the proof of
  Proposition~\ref{prop:markov_forms_geo} to show that there is a
  sequence of horocycles $(h_{n})$ and an increasing sequence of
  Markov numbers $(c_{n})$ such that
  $d(g,h_{n})<-\frac{1}{2}\log\big(\frac{9}{4}-\frac{1}{c_{n}^{2}}\big)$.
  This implies that (ii) is false.
\end{proof}

\section{Dictionary: point --- definite form. Spectrum, classification
  of definite forms, and the Farey tessellation revisited}
\label{sec:point}

This section is about the hyperbolic geometry of definite binary
quadratic forms. Its purpose is to complete the dictionary and provide
a broader perspective. This section is not needed for the proof of
Markov's theorem.

If the binary quadratic form~\eqref{eq:f} with real coefficients is
positive or negative definite, then the polynomial~$f(x,1)$ has two
complex conjugate roots. Let $z(f)$ denote the root in the upper
half-plane, i.e.,
\begin{equation*}
  z(f)=\frac{-B+i\sqrt{\det f}}{A}\,.
\end{equation*}
This defines a map $f\mapsto z(f)$ from the space of definite forms to
the hyperbolic plane $H^{2}$. It is surjective and many-to-one (any
non-zero multiple of a form is mapped to the same point) and
equivariant with respect to the left $\GLTR$-actions.

The \emph{signed distance} of a horocycle and a point in the
hyperbolic plane is defined in the obvious way (positive for points
outside, negative for points inside the horocycle). One obtains the
following proposition in the same way as the corresponding statement
about geodesics (Proposition~\ref{prop:dhgf}):

\begin{proposition}
  \label{prop:dhzf}
  The signed distance of the horocycle $h(p,q)$ and the
  point $z(f)\in H^{2}$ is
  \begin{equation}
    \label{eq:dhzf}
    d\big(h(p,q),z(f)\big)=\log\frac{|f(p,q)|}{\sqrt{\det f}}\,.
  \end{equation}
\end{proposition}

This provides a geometric explanation for the different behavior of
definite binary quadratic forms with respect to their minima on
$\Z^{2}$: 

For all definite forms $f$, the infimum~\eqref{eq:mf} is attained for
some $(p,q)\in\Z^{2}$ and satisfies $M(f)\leq\frac{2}{\sqrt{3}}$. All
forms equivalent to $p^{2}-pq+q^{2}$, and only those, satisfy
$M(f)=\frac{2}{\sqrt{3}}$. But for every positive number
$m<\frac{2}{\sqrt{3}}$, there are infinitely many equivalence classes
of definite forms with $M(f)=m$.

Algorithms to determine the minimum $M(f)$ of a definite quadratic
form $f$ are based on the reduction theory for quadratic forms. (The
theory of equivalence and reduction of binary quadratic forms is
usually developed for integer forms, but much of it carries over to
forms with real coefficients.) The reduction algorithm described by
Conway~\cite{conway97} has a particularly nice geometric
interpretation based on the following observation:

For a point in the hyperbolic plane, the three nearest Ford circles
(in the sense of signed distance) are the Ford circles at the vertices
of the Farey triangle containing the point. (If the point lies on an
edge of the Farey tessellation, the third nearest Ford circle is not
unique.)
   
\paragraph{Acknowledgement.} I would like to thank Oliver Pretzel, who
gave me a first glimpse of this subject some 25 years ago, and
Alexander Veselov, who made me look again. Last but not least, I would
like to thank the anonymous referees for their insightful comments.

This research was supported by DFG SFB/TR 109 ``Discretization in Geometry
and Dynamics''.

\begingroup
\small
\bibliographystyle{abbrv}
\bibliography{markov}
\endgroup

\vspace{3\baselineskip}\noindent%
Boris Springborn\\
Technische Universit{\"a}t Berlin\\
Institut f{\"u}r Mathematik, MA 8-3\\
Str.~des 17.~Juni 136\\
10623 Berlin, Germany\\[\baselineskip]
\nolinkurl{boris.springborn@tu-berlin.de}

\end{document}

%% file: markovtree.tex
\begin{tikzpicture}[grow cyclic]
  \tikzstyle{level 1}=[level distance=25pt,sibling angle=120] 
  \tikzstyle{level 2}=[level distance=25pt,sibling angle=80] 
  \tikzstyle{level 3}=[level distance=20pt,sibling angle=60] 
  \tikzstyle{level 4}=[level distance=15pt,sibling angle=60] 
  \tikzstyle{level 5}=[level distance=15pt,sibling angle=60] 
  \tikzstyle{every node}=[font=\small]
  \coordinate
  child foreach \x in {1,2,3}
  {child foreach \x in {1,2}
    {child foreach \x in {1,2}
      {child foreach \x in {1,2}
        {child foreach \x in {1,2}}}}};
  
  \draw (-35pt,0pt) node {1};
  \draw[rotate=120] (-35pt,0pt) node {1};
  \draw[rotate=-120] (-35pt,0pt) node {1};
  
  \draw (58pt,0pt) node {2};
  \draw[rotate=120] (58pt,0pt) node {2};
  \draw[rotate=-120] (58pt,0pt) node {2};
  
  \draw (62pt,31pt) node {5};
  \draw (62pt,-31pt) node {5};
  \draw[rotate=120] (62pt,31pt) node {5};
  \draw[rotate=120] (62pt,-31pt) node {5};
  \draw[rotate=-120] (62pt,31pt) node {5};
  \draw[rotate=-120] (62pt,-31pt) node {5};
  
  \draw (81pt,23pt) node {29};
  \draw (81pt,-23pt) node {29};
  \draw[rotate=120] (81pt,23pt) node {29};
  \draw[rotate=120] (81pt,-23pt) node {29};
  \draw[rotate=-120] (81pt,23pt) node {29};
  \draw[rotate=-120] (81pt,-23pt) node {29};
  
  \draw (56.5pt,51pt) node {13};
  \draw (56.5pt,-51pt) node {13};
  \draw[rotate=120] (56.5pt,51pt) node {13};
  \draw[rotate=120] (56.5pt,-51pt) node {13};
  \draw[rotate=-120] (56.5pt,51pt) node {13};
  \draw[rotate=-120] (56.5pt,-51pt) node {13};
  
  \draw (95pt,10pt) node {169};
  \draw (95pt,-10pt) node {169};
  \draw[xshift=0.5pt,yshift=0pt] [rotate=120] (95pt,10pt) node {169};
  \draw[xshift=-2pt,yshift=-3pt] [rotate=120] (95pt,-10pt) node {169};
  \draw[xshift=0pt,yshift=4pt] [rotate=-120] (95pt,10pt) node {169};
  \draw[xshift=1pt,yshift=0pt] [rotate=-120] (95pt,-10pt) node {169};
  
  \draw (89pt,38pt) node {433};
  \draw (89pt,-38pt) node {433};
  \draw[xshift=-1pt,yshift=-3pt] [rotate=120] (89pt,38pt) node {433};
  \draw[xshift=2.5pt,yshift=0pt] [rotate=120] (89pt,-38pt) node {433};
  \draw[xshift=2pt,yshift=0pt] [rotate=-120] (89pt,38pt) node {433};
  \draw[xshift=0pt,yshift=3.5pt] [rotate=-120] (89pt,-38pt) node {433};
  
  \draw (76pt,53pt) node {194};
  \draw (76pt,-53pt) node {194};
  \draw[xshift=-2pt,yshift=-3pt] [rotate=120]  (76pt,53pt) node {194};
  \draw[xshift=2pt,yshift=1pt] [rotate=120]  (76pt,-53pt) node {194};
  \draw[xshift=2.5pt,yshift=-1pt] [rotate=-120]  (76pt,53pt) node {194};
  \draw[xshift=-1pt,yshift=4pt] [rotate=-120]  (76pt,-53pt) node {194};
  
  \draw (46pt,63pt) node {34};
  \draw (46pt,-63pt) node {34};
  \draw (32pt,71pt) node {34};
  \draw (32pt,-71pt) node {34};
  \draw (-79pt,8pt) node {34};
  \draw (-79pt,-8pt) node {34};
  
  \draw (-5pt, 0pt) node {\scriptsize $a$};
  \draw (2.5pt, 5pt) node {\scriptsize $b$};
  \draw (2pt, -4pt) node {\scriptsize $c$};
  
  \draw (32pt, 1.75pt) node {\scriptsize $a'$};
  \draw (23.5pt, 5pt) node {\scriptsize $b$};
  \draw (23.5pt, -4pt) node {\scriptsize $c$};
  
  \draw (-15.5pt, 17.5pt) node {\scriptsize $a$};
  \draw (-8pt, 23pt) node {\scriptsize $b$};
  \draw (-15pt, 28.5pt) node {\scriptsize $c'$};
  
  \draw (-16pt, -18pt) node {\scriptsize $a$};
  \draw (-8pt, -22pt) node {\scriptsize $c$};
  \draw (-17pt, -29pt) node {\scriptsize $b'$};
\end{tikzpicture}  
%%% Local Variables:
%%% TeX-master: "markov"
%%% End: